\documentclass[11pt]{amsart}
\usepackage{cases,amsmath,amssymb, hyperref}
\theoremstyle{plain}
\newtheorem{theorem}                {Theorem}      [section]
\newtheorem*{theorem*}                {Theorem \ref{thm:appl}}
\newtheorem{proposition}  [theorem]  {Proposition}
\newtheorem{corollary}    [theorem]  {Corollary}
\newtheorem{lemma}        [theorem]  {Lemma}

\theoremstyle{definition}

\newtheorem{remark}       [theorem]  {Remark}
\newtheorem{definition}   [theorem]  {Definition}

\DeclareMathOperator{\trace}{trace} 

\DeclareMathOperator{\Div}{div} 
 \DeclareMathOperator{\riem}{Riem}

\DeclareMathOperator{\Span}{span}

 \DeclareMathOperator{\im}{Im}
\DeclareMathOperator{\grad}{grad}

\numberwithin{equation}{section}

\begin{document}

\thanks{This work was supported by a grant of the Romanian Ministry of Research and Innovation,
CCCDI-UEFISCDI, project number PN-III-P3-3.1-PM-RO-FR-2019-0234 / 1BM / 2019, within
PNCDI III and the PHC Brancusi 2019 project no 43460 TL}

\address{Univ. Brest, CNRS UMR 6205, LMBA, F-29238 Brest, France}\email{Hiba.Bibi@etudiant.univ-brest.fr}

\address{Department of Mathematics, Michigan State University, East Lansing, MI 48824–
1027, USA}
\email{bychen@math.msu.edu}

\address{Department of Mathematics and Informatics, Gh. Asachi Technical University of
Iasi, Bd. Carol I, no. 11, 700506 Iasi, Romania}
\email{dorel.fetcu@etti.tuiasi.ro}

\address{Faculty of Mathematics\\ Al. I. Cuza University of Iasi\\
Bd. Carol I, no. 11 \\ 700506 Iasi, Romania} \email{oniciucc@uaic.ro}



\title[PMC biconservative surfaces]{PMC  biconservative surfaces in complex space forms}

\author{Hiba Bibi, Bang-Yen Chen, Dorel Fetcu, Cezar Oniciuc}

\maketitle

\begin{abstract}
In this article we consider PMC surfaces in complex space forms, and we study the interaction between the notions of PMC, totally real and biconservative.
We first consider PMC surfaces in non-flat complex space forms  and we prove that they are biconservative if and only if totally real. Then, we find a Simons type formula  for a well-chosen vector field constructed from the mean curvature vector field. Next, we prove a rigidity result for CMC biconservative surfaces  in $2$-dimensional complex space forms. We prove then a reduction codimension result for PMC biconservative surfaces in non-flat complex space forms. We conclude by constructing from the Segre embedding examples of CMC non-PMC  biconservative submanifolds, and we also discuss when they are proper-biharmonic.
\end{abstract}

\section{Introduction}
Biharmonic submanifolds of Euclidean spaces $\mathbb{E}^{n}$ were introduced in the middle of the 1980s by B.-Y. Chen \cite{43}  as isometric immersions with harmonic mean curvature vector field (see also \cite{Dimitric}), and in \cite{Ishikawa, J2}, they proved that biharmonic surfaces in
$\mathbb{E}^{3}$ are minimal. This led to conjecture that biharmonic submanifolds of Euclidean spaces are minimal (see \cite{43}). Several partial results support B.-Y. Chen's conjecture \cite{Akutagawa and Maeta, Dimitric,Fu-Hong-Zhan, Hasanis, M-O-R-Euclidean}.

Independently, and almost contemporarily, the biharmonicity was   defined in the mid-80's at a more abstract level of a variational problem for maps between Riemannian manifolds by G.-Y. Jiang \cite{Jiang 2009, Jiang}, which shows  biharmonic maps as critical points of the $L^{2}-$norm of the tension field functional. Since ambient spaces with non-positive curvatures do not admit interesting compact examples, most research has been done on biharmonic submanifolds of Euclidean spheres (see, for example, \cite{Fetcu-Oniciuc-Survey, Habilitation thesis, Chen-Ou}). Many properties of the biharmonic submanifold of the Euclidean spheres follow from the fact that, when having constant mean curvature, they are $2$-type submanifolds of the ambient Euclidean space, in the sense of B.-Y. Chen \cite{B-Y. Chen}. As a next step, the biharmonicity in complex projective spaces was studied in \cite{F-L-M-O, Fetcu-Pinheiro, I-I-U, Sasahara-2019}.

From the theory of biharmonic submanifolds,  the study of biconservative submanifolds is derived, as such submanifolds are characterized  by the vanishing of the tangential part of the bitension field. 
By studying biconservative submanifolds we try to check how much we can rely on just one part of the biharmonic equation, and inspect what results can still be valid with this condition. Biconservative submanifolds were studied in \cite{C-M-O-P, F-O-P, LMO, Manfio-Turgay, M-O-R-Euclidean,  M.O.R., Sasahara-2015, Sasahara-surfaces, Turgay, T.U.}. 

This article starts with proving that PMC surfaces in the non-flat complex space forms are biconservative if and only if totally real. Then we develop a Simons type formula  for a well-chosen vector field constructed from the mean curvature vector field, to deduce that a complete PMC totally real surface of a non-negative  Gaussian curvature in a complex space form  must have parallel shape operator. This surface must be either flat or pseudo-umbilical. Next, restricting ourselves to complex space forms of complex dimension $2$, we find optimal conditions so that a CMC biconservative surface must be PMC. Then, we use the reduction techniques of codimension  used in \cite{Alencar, Eschenburg-Tribuzy, Ferreira-Tribuzy} to improve results in \cite{dorel, Fetcu-Pinheiro}, and show that a non pseudo-umbilical PMC biconservative surface in a non-flat complex space form $N^{n}(c)$ must lie in some $N^{4}(c)\subset N^{n}(c)$. One particular case further reduces the real codimension more to $2$. We conclude using the Segre embedding to construct an example of CMC biconservative submanifolds $M^{1+2q}$ of the complex projective space 
$\mathbb{C}P^{1+2q}(4)$, which are neither PMC  nor totally real. Moreover, we discuss their biharmonicity. This illustrates having higher dimension of a biconservative submanifold and getting a less rigid conclusion, more interesting examples than the PMC or  totally real ones may be found in the future.

\textbf{Conventions.}
Throughout this paper, a surface $M^{2}$ means an oriented manifold of real dimension $2$. For an arbitrary Riemannian manifold, the metric will be indicated by 
$\langle  , \rangle$, or simply omitted, and the following sign conventions will be used
$$
R(X,Y)Z=\nabla_{X}\nabla_{Y}Z-\nabla_{Y}\nabla_{X}Z-\nabla_{[X,Y]}Z,
$$
and 
$$
\Delta=-\trace\nabla^{2}.
$$
A complex space form of complex dimension $n$ and constant holomorphic sectional curvature $c$ will be denoted by $N^{n}(c)$.
Since a minimal submanifold is trivially  biconservative, we will always assume that the mean curvature vector field $H$ does not vanish. When dealing with a submanifold $M$ of $N$, we will indicate the objects on the target manifold $N$ by $\overline{(\cdot)}$.

\textbf{Acknowledgements.}
The authors would like to thank Katsuei Kenmotsu and Eric Loubeau for carefully reading our paper and for their comments and suggestions.

\section{Preliminaries}

A biharmonic map $\phi : M^{m} \to N^{n}$ between two fixed Riemannian manifolds is a critical point of the bienergy functional
$$
E_{2}:C^{\infty}(M,N)\to \mathbb{R}, \quad E_{2}(\phi)=\frac{1}{2}\int_{M} \vert \tau(\phi)\vert ^{2} dv,
$$
where $M$ is compact and $ \tau(\phi)= \trace  \nabla d \phi$ is the tension field of $\phi$. These maps are characterized by the Euler-Lagrange equation, also known as the biharmonic equation, obtained by G.-Y. Jiang in 1986 (see \cite{Jiang}):
\begin{eqnarray}\label{tau-2}
\tau_{2}(\phi)=-\Delta \tau(\phi)-\trace \overline{R}(d\phi(\cdot),\tau(\phi))d \phi (\cdot)=0,
\end{eqnarray}
where $\tau_{2}(\phi)$ is the bitension field of $\phi$.

As any harmonic map is biharmonic, we are interested in studying non-harmonic biharmonic  maps, which are called proper-biharmonic maps.

We fix a map $\phi$ and now we let the domain metric to vary. We obtain a new functional on the set $\mathcal{G}$ of all
Riemannian metrics on $M^{m}$ defined by 
$$
\mathcal{F}_{2}:\mathcal{G}\to \mathbb{R}, \quad \mathcal{F}_{2}(g)=E_{2}(\phi).
$$
Critical points of this functional are characterized by the vanishing of the stress-energy
tensor of the bienergy (see \cite{LMO}). This tensor, denoted by $S_{2}$, was introduced
in \cite{Jiang87} as
\begin{eqnarray*}
S_{2}(X,Y)&=&\frac{1}{2}\vert \tau (\phi)\vert ^{2}\langle X,Y \rangle +\langle d \phi, \nabla \tau (\phi) \rangle \langle X, Y \rangle-\langle d\phi (X),\nabla_{Y} \tau (\phi)\rangle
\\
&\ & -\langle d\phi (Y),\nabla_{X} \tau (\phi)\rangle,
\end{eqnarray*}
and it satisfies
$$
\Div S_{2}=\langle \tau_{2}(\phi), d\phi\rangle.
$$

We note that, for isometric immersions, $(\Div S_{2})^{\sharp} =-\tau_{2}(\phi)^{\top}$, where $\tau_{2}(\phi)^{\top}$ is the
tangent part of the bitension field.

\begin{definition}
A submanifold 
$\phi:M^{m} \to N^{n}$
of a Riemannian manifold $N^{n}$ is called \emph{biconservative} if $\Div S_{2}=0$. 
\end{definition}

In general, for a submanifold 
$\phi:M^{m} \to N^{n}$ we will not mention explicitly the isometric immersion $\phi$ and we will simply say that $M^{m}$ is a submanifold of $N^{n}$. For the sake of simplicity, we recall here the fundamental equations of a submanifold. The Gauss Equation: 
\begin{eqnarray}\label{the Gauss Equation}
\langle \overline{R}(X,Y)Z,W\rangle &=& \langle R(X, Y)Z,W\rangle +\langle B(X,Z),B(Y,W)\rangle \nonumber
\\
&\ &-\langle B(X,W),B(Y,Z)\rangle,
\end{eqnarray}
where $X$, $Y$, $Z$ and $W$ are vector fields tangent to $M^{m}$, and $B$ the second fundamental form of $M^{m}$ in $N^{n}$.
\\
The Codazzi Equation:
\begin{eqnarray}\label{the Codazzi Equation}
(\nabla_{X}^{\perp}B)(Y,Z)-(\nabla_{Y}^{\perp}B)(X,Z)=(\overline{R}(X,Y)Z)^{\perp},
\end{eqnarray}
where 
$$
(\nabla_{X}^{\perp}B)(Y,Z)=\nabla^{\perp}_{X}B(Y,Z)-B(\nabla_{X}Y,Z)-B(Y,\nabla_{X}Z).
$$
Here, $\nabla^{\perp}$ is the connection in the normal bundle $NM^{m}$ of $M^{m}$ in $N^{n}$ and $\nabla$ is the Levi-Civita connection of $M^{m}$.
\\
The Ricci Equation:
\begin{eqnarray}\label{the Ricci Equation}
\langle R^{\perp}(X,Y)U,V\rangle=\langle[A_{U},A_{V}]X,Y\rangle+\langle \overline{R}(X,Y)U,V\rangle,
\end{eqnarray}
where $U$ and $V$ are vector fields normal to $M^{m}$, and $A$ denotes the shape operator.

\begin{definition}
Let $M^{m}$ be a submanifold of a Riemannian manifold  $N^{n}$. If the mean curvature vector field $H$ of $M^{m}$ is parallel in the normal bundle, i.e., $\nabla^{\perp} H=0$, then $M^{m}$ is called a \emph{PMC submanifold}.
\end{definition}

\begin{proposition}\label{Biconservativity Cdts}(\cite{LMO, Nistor1})
Let $M^{m}$ be a submanifold of a Riemannian manifold $N^{n}$. Then the following properties are equivalent:
\begin{enumerate}
\item $M$ is biconservative;
\item  $\trace A_{\nabla_{(\cdot)}^{\perp}H}(\cdot)+\trace(\nabla A_{H})(\cdot, \cdot)+\trace(\overline{R}(\cdot,H)\cdot)^{\top}=0$;
\item $4\trace A_{\nabla_{(\cdot)}^{\perp}H}(\cdot)+m \grad(\vert H \vert ^{2})+4\trace(\overline{R}(\cdot,H)\cdot)^{\top}=0$;
\item $4\trace(\nabla A_{H})(\cdot, \cdot)-m\grad (\vert H \vert ^{2})=0.$
\end{enumerate}
\end{proposition}

As an immediate consequence we get

\begin{proposition}\label{PMC-biconservative}
Let $M^{m}$ be a PMC submanifold of a Riemannian manifold $N^{n}$. Then $M^{m}$ is biconservative if and only if 
$$
\trace(\overline{R}(\cdot,H)\cdot)^{\top}=0.
$$
\end{proposition}

When the ambient space is a space form, i.e. it has a constant Gaussian curvature, we have 

\begin{corollary}(\cite{F-O-P})
Let $M^{m}$ be a PMC submanifold of a real space form $N^{n}$. Then $M^{m}$ is biconservative.
\end{corollary}

\begin{definition}
A submanifold $M^{m}$ of the complex manifold $N$  equipped with the complex structure $J$ is said to be \emph{totally real} if $JTM^{m}$ lies in the normal bundle of $M^{m}$.
\end{definition}

We recall that the curvature tensor field 
$\overline{R}$ of a complex space form $N^{n}(c)$ of complex dimension $n$ is given by
\begin{eqnarray}\label{Tensorial Curvature}
\overline{R}(\overline{X},\overline{Y})\overline{Z} &=&\frac{c}{4}\Big\{\langle \overline{Y},\overline{Z}\rangle \overline{X}-\langle \overline{X},\overline{Z}\rangle\overline{Y} +\langle J\overline{Y},\overline{Z}\rangle J \overline{X}-\langle J \overline{X},\overline{Z}\rangle J\overline{Y}\nonumber
\\
&\ & \quad  +2\langle J\overline{Y},\overline{X}\rangle J\overline{Z}\Big\},
\end{eqnarray}
where 
$\overline{X}$, $\overline{Y}$ and 
$\overline{Z}$ are vector fields tangent to  $N$.

We denote $JH=T+N$, $T$ being the tangential part of $JH$ and $N$ the  normal part of $JH$, i.e., $T=(JH)^{\top}$ and $N=(JH)^{\perp}$.

\begin{theorem}\label{PMC JT}
Let $M^{m}$  be a PMC submanifold of a complex space form $N^{n}(c)$. If $c=0$, then $M^{m}$ is biconservative, and if $c\neq 0$, then $M^{m}$ is biconservative if and only if $JT \in  C(NM^{m})$.
\end{theorem}

\begin{proof}
Clearly, from  Equation ~(\ref{Tensorial Curvature}), one can see that for a PMC submanifold $M^{m}$,  the  biconservativity condition (see Proposition \ref{PMC-biconservative})
$$
\trace (\overline{R}(\cdot, H)\cdot)^{\top}=\frac{3}{4}c(JT)^{\top}=0
$$
holds if and only if  either $c=0$, or $JT\in C(NM^{m})$.
\end{proof}

\begin{corollary}
Let $M^m$ be a PMC totally real submanifold of a complex space form $N^{n}(c)$. Then $M^m$ is biconservative.
\end{corollary}

\begin{corollary}
Every PMC real hypersurface $M^{2n-1}$ of a complex space form
$N^{n}(c)$ is biconservative.
\end{corollary}

\begin{proof}
In view of Theorem \ref{PMC JT}, we may assume that $c \neq 0$. Since the mean curvature vector $H$ is normal to $M^{2n-1}$ in 
$N^{n}(c)$, and the codimension is one, we have $JH = T$, hence $JT =-H$ is a normal vector field. Therefore, $JT \in C(NM^{2n-1})$. Consequently, the real hypersurface $M^{2n-1}$ is always biconservative.
\end{proof}

Other sufficient conditions for biconservativity are given by the following results.

\begin{theorem}
\label{PMC-biconservative-submanifold}Let $M^{m}$ be a PMC submanifold of a complex space
form $N^{n}(c)$ with $c \neq 0$. If $JH \in C(NM^{m})$, then $M^{m}$ is biconservative.
\end{theorem}

\begin{proof}
If $JH \in C(NM^{m})$, then $T = JT = 0$. Hence $M^{m}$ is biconservative.
\end{proof}

\begin{theorem}
Let $M^{m}$ be a PMC submanifold of a complex space
form $N^{n}(c)$ with $c \neq 0$. If $JH \in C(TM^{m})$, then $M^{m}$ is biconservative.
\end{theorem}

\begin{proof}
If $JH \in C(TM^{m})$, then $T = JH$ and so 
$JT=-H$ is normal. Hence $M^{m}$ is biconservative.
\end{proof}

\section{PMC biconservative surfaces in $N^{n}(c)$}

In this section we study PMC surfaces in the complex space form $N^{n}(c)$ of complex dimension $n$.

\begin{theorem}\label{Theorem}
Let $M^{2}$  be a PMC surface in a complex space form $N^{n}(c)$. If $c=0$, then $M^{2}$ is biconservative, and if $c\neq 0$, then $M^{2}$ is biconservative if and only if $M^{2}$ is totally real.
\end{theorem}

\begin{proof}
The case $c=0$ can be  easily proved. Further, we will consider the case $c\neq 0$. 

First, we will prove that a PMC biconservative surface $M^{2}$ in $N^{n}(c)$ is totally real. Since $M^{2}$ is PMC and biconservative, from Theorem \ref{PMC JT} we have 
\begin{eqnarray}
(JT)^{\top}=0.
\end{eqnarray}

Now, from the Ricci equation \eqref{the Ricci Equation}, 
since $M^{2}$ is PMC and taking $U=H$, we obtain
$$
\langle[A_{H},A_{V}]X,Y\rangle=-\langle \overline{R}(X,Y)H,V\rangle.
$$
Now, using  (\ref{Tensorial Curvature})
\begin{eqnarray*}
\langle\overline{R}(X,Y)H,JT\rangle &=&\frac{c}{4}\Big \{\langle Y,H\rangle\langle X,JT\rangle-\langle X,H\rangle\langle Y,JT\rangle +\langle JY,H\rangle \langle JX,JT \rangle
\\
&\ & \quad -\langle JX,H\rangle \langle JY,JT\rangle
+2\langle JY, X\rangle \langle JH,JT\rangle \Big\}
\\
&=&
\frac{c}{4}\Big \{\langle JY,H\rangle \langle JX,JT \rangle
-\langle JX,H\rangle \langle JY,JT\rangle
\\
&\ & \quad +2\langle JY, X\rangle \langle JH,JT\rangle \Big\}
\\
&=&
\frac{c}{4}\Big \{-\langle JH,Y\rangle \langle X,T \rangle
+\langle JH,X\rangle \langle Y,T\rangle
\\
&\ & \quad -2\langle JX, Y\rangle \langle H,T\rangle \Big\}
\\
&=&\frac{c}{4}\Big \{-\langle T,Y\rangle \langle X,T \rangle
+\langle T,X\rangle \langle Y,T\rangle
\Big\}
\\
&=&0,
\end{eqnarray*}
then 
\begin{eqnarray}\label{Lie Bracket}
[A_{H}, A_{JT}]=0.
\end{eqnarray}
Now, from Equation (\ref{Lie Bracket}) it follows that at each point of $M^{2}$ there exists a (positive) orthonormal basis $\{e_{1},e_{2}\}$ tangent to $M^{2}$ that diagonalizes both $A_{H}$ and $A_{JT}$ at that point.

Moreover, the following equality holds on $M^{2}$
\begin{eqnarray*}
\trace A_{JT} &=&\sum_{i=1}^{2}\langle A_{JT}e_{i},e_{i}\rangle
=
\sum_{i=1}^{2}\langle B(e_{i},e_{i}), JT\rangle
=\langle \trace B,JT\rangle
\\&=& 2\langle H,JT\rangle
=-2\langle JH,T\rangle
=
-2\langle T,T\rangle
\\
&=&
-2\vert T \vert ^{2}.
\end{eqnarray*}

Further, as $\nabla ^{\perp}H=0$, we have
\begin{eqnarray*}
\overline{\nabla}_{X}JH &=&J\overline{\nabla}_{X}H
=
J(\nabla_{X}^{\perp}H-A_{H}X)
\\
&=&
-JA_{H}X,
\end{eqnarray*}
and 
\begin{eqnarray*}
\overline{\nabla}_{X}JH&=&\overline{\nabla}_{X}(T+N)
\\
&=& \nabla_{X}T+B(X,T)-A_{N}X+\nabla_{X}^{\perp}N,
\end{eqnarray*}
thus 
\begin{eqnarray}\label{Equality}
-JA_{H}X &=& \nabla_{X}T+B(X,T)-A_{N}X+\nabla_{X}^{\perp}N.
\end{eqnarray}
We fix a point $p$ and then, for $X=e_{i}$, we take the inner product of (\ref{Equality}) with $e_{j}$, $i \neq j$, and at $p$ we obtain
\begin{eqnarray*}
-\langle JA_{H}e_{i},e_{j}\rangle &=& \langle \nabla_{e_{i}}T,e_{j}\rangle-\langle A_{N}e_{i}, e_{j}\rangle +\langle B(e_{i},T),e_{j} \rangle +\langle \nabla_{e_{i}}^{\perp}N,e_{j}\rangle
\\
&=&
\langle \nabla_{e_{i}}T,e_{j}\rangle-\langle A_{N}e_{i}, e_{j}\rangle.
\end{eqnarray*}
With respect to the basis $\{e_{1}, e_{2} \}$, we have
$$
A_{H}=
\begin{bmatrix} 
\lambda_{1} & 0  \\
0 & \lambda_{2}
\end{bmatrix},\quad 
A_{JT}=
\begin{bmatrix} 
\mu_{1} & 0 \\
0 & \mu_{2}
\end{bmatrix}.
\quad
$$
Thus, at $p$,
\begin{eqnarray}\label{Eigenvalues equality}
-\lambda_{i}\langle Je_{i},e_{j}\rangle =\langle \nabla_{e_{i}}T, e_{j}\rangle-\langle A_{N}e_{i},e_{j}\rangle.
\end{eqnarray}

One can see that 
$$
\lambda_{1}+\lambda_{2}=\trace {A_{H}}=2 \vert H \vert ^{2}
$$
and
$$
\mu_{1}+\mu_{2}= \trace A_{JT}=-2\vert T \vert ^{2}. 
$$
On the other hand, on $M^{2}$,
\begin{eqnarray*}
\overline{\nabla}_{X}JT&=&-A_{JT}X+\nabla_{X}^{\perp}JT,
\end{eqnarray*}
and
\begin{eqnarray*}
\overline{\nabla}_{X}JT=J\overline{\nabla}_{X}T
=J\nabla_{X}T+JB(X,T).
\end{eqnarray*}
For $X=e_{i}$ in the above relations, at the point $p$, we get
\begin{eqnarray*}
-\langle A_{JT}e_{i}, e_{i}\rangle &=&\langle J \nabla _{e_{i}}T,e_{i}\rangle +\langle JB(e_{i},T),e_{i}\rangle
\\
&=&-\langle \nabla_{e_{i}}T,Je_{i}\rangle - \langle B(e_{i},T),Je_{i} \rangle,
\end{eqnarray*}
and therefore,
$$
-\langle \mu_{i}e_{i},e_{i}\rangle = -\langle \nabla_{e_{i}}T,Je_{i}\rangle -\langle B(e_{i},T),Je_{i}\rangle,
$$
which implies
\begin{eqnarray}
\mu_{i}&=& \sum_{j=1}^{2}\langle \nabla_{e_{i}}T,e_{j}\rangle \langle Je_{i},e_{j}\rangle +\langle B(e_{i},T),Je_{i}\rangle.\nonumber
\\
&=&\langle \nabla_{e_{i}}T,e_{j}\rangle \langle Je_{i},e_{j}\rangle +\langle B(e_{i},T),Je_{i}\rangle, \quad i\neq j.\label{equality of mu}
\end{eqnarray}

Now, we multiply Equation (\ref{Eigenvalues equality}) by $\langle Je_{i},e_{j}\rangle $, $i \neq j$, to obtain
\begin{eqnarray}\label{equality of lambda}
-\lambda_{i}\langle Je_{i}, e_{j}\rangle ^{2} &=&\langle \nabla_{e_{i}}T,e_{j}\rangle\langle Je_{i},e_{j}\rangle - \langle A_{N}e_{i},e_{j}\rangle\langle Je_{i},e_{j}\rangle.
\end{eqnarray}
From  Equations (\ref{equality of mu}) and (\ref{equality of lambda}), we get 
$$
-\lambda_{i}\langle Je_{i}, e_{j}\rangle ^{2}=\mu_{i}-\langle B(e_{i},T),Je_{i}\rangle - \langle A_{N}e_{i},Je_{i}\rangle.
$$
Thus, by summing up, we have
\begin{eqnarray}\label{Sum over i}
 \mu_{1} +\mu_{2}&=& - (\lambda_{1}+ \lambda_{2})\langle Je_{1}, e_{2}\rangle ^{2}+ \sum_{i=1}^{2} \langle B(e_{i},T),Je_{i}\rangle +\sum_{i=1}^{2}\langle A_{N}e_{i},Je_{i}\rangle. 
\end{eqnarray}
Since
\begin{eqnarray}
\langle A_{N}e_{1},Je_{1}\rangle &=&\langle A_{N}e_{1},e_{2}\rangle  \langle e_{2},Je_{1}  \rangle \nonumber
=-\langle e_{1},A_{N}e_{2}\rangle  \langle Je_{2},e_{1}  \rangle\nonumber
\\
&=& -
\langle A_{N}e_{2},Je_{2}\rangle, \nonumber
\end{eqnarray}
then
\begin{eqnarray}
\mu_{1}+\mu_{2}&=& -(\lambda_{1}+\lambda_{2})\langle Je_{1}, e_{2}\rangle ^{2}+ \sum_{i=1}^{2} \langle B(e_{i},T),Je_{i}\rangle \nonumber
\\ 
2 \vert T \vert ^{2}&=& 2 \vert H \vert ^{2}\langle Je_{1}, e_{2}\rangle ^{2} + \trace \langle JB(\cdot,T),\cdot\rangle, \label{Equality of |T|}
\end{eqnarray}
which holds at any point $p \in M^{2}$.
We note that Equation (\ref{Equality of |T|}) has a geometrical meaning.

Now, let $p \in M^{2}$  be an arbitrary point. If $T_{p}\neq 0$, we can consider the orthonormal  basis $\{ X_{1}, X_{2}\}$, where  $X_{1}=T_{p}/\vert T_{p} \vert$, tangent to $M^{2}$. Then, since  $JT$ is normal, we have $\langle JX_{2}, X_{1}\rangle=0$, and it is easy to see that \begin{eqnarray*}
J(T_{p}M^{2})\subset N_{p}M^{2}.
\end{eqnarray*}
Now assume that $T_{p}=0$. From (\ref{Equality of |T|}), it follows that  
$$
2 \vert H \vert ^{2} \langle Je_{1},e_{2}\rangle ^{2}=0,
$$
that is $\langle Je_{1},e_{2}\rangle=0$, which shows that  
\begin{eqnarray*}
J(T_{p}M^{2}) \subset N_{p}M^{2}.
\end{eqnarray*}
Therefore, we conclude that $M^{2}$ is totally real.

Conversely, if $M^{2}$ is totally real, then $(JT)^{\top}=0$ and therefore, $\trace (\overline{R}(\cdot,H), \cdot)^{\top}=0$. Since $M^{2}$ is also PMC,  from Proposition \ref{PMC-biconservative}, it follows that $M^{2}$ is biconservative. 
\end{proof}

\begin{remark}
We note that if $M^{2}$  is a PMC surface in a complex space form 
$N^{n}(c)$ and if $JH \in C(NM^{2})$, i.e. $T=0$,  then $M^{2}$ is totally real.
\end{remark}

\begin{remark}
For $c=0$, every PMC submanifold of a complex $n$-dimensional Euclidean space 
$\mathbb{C}^{n}$ 
is biconservative, but not
necessarily totally real. For instance, $\mathbb{S}^{2}(1) \subset \mathbb{E}^{3} \subset \mathbb{C}^{2}$
is PMC and biconservative
in $\mathbb{C}^{2}$ but not totally real, where $\mathbb{E}^{3}$ is the real $3$-dimensional Euclidean space.
\end{remark}

Concerning slant surfaces (see \cite{Chen-slant}), we have the following non-existence result, which is a direct application of Theorem \ref{Theorem}.

\begin{corollary}
Every PMC proper slant surface in a non-flat complex space form 
$N^{n}(c)$ is not biconservative.
\end{corollary}

We note that from the proof of Theorem \ref{Theorem} we get the following general result.

\begin{theorem}
Let $M^{2}$  be a PMC surface in a complex space form $N^{n}(c)$. Then $JT\in C(NM^{2})$ if and only if $M^{2}$ is totally real.
\end{theorem}

Now we recall the following result that holds for surfaces:

\begin{theorem} \cite{Nistor1}\label{Nistor}
Let $M^{2}$ be a complete CMC biconservative surface in a Riemannian manifold $N^{n}$. Assume that $K\geq 0$ and $\riem^{N}\leq K_{0}$, where $K_{0}$ is a constant. Then $\nabla A_{H}=0$, and either $M^{2}$ is flat or pseudo-umbilical.
\end{theorem}

Since, by Theorem \ref{Theorem}, a PMC totally real surface in a complex space form $N^{n}(c)$  is biconservative, and a complex space form has the Gaussian curvature bounded by $c/4$ and $c$, we get

\begin{corollary}\label{Corollary on complete PMC}
Let $M^{2}$ be a complete PMC totally real surface  with $K\geq 0$ in a complex space form $N^{n}(c)$. Then $\nabla A_{H}=0$, and either $M^{2}$ is flat or pseudo-umbilical.
\end{corollary}

\begin{remark}
The result in Corollary \ref{Corollary on complete PMC} is similar to \cite[Theorem 5.4]{Chen-Ogiue}, but here $M^{2}$ is complete, not necessarily compact, and the proof of Theorem \ref{Nistor} relies on a different technique compared to that used in the proof of Theorem 5.4.
\end{remark}

However, PMC totally real surfaces in complex space forms have more specific properties.

\begin{theorem}
Let $M^{2}$ be a PMC totally real surface in the complex space form $N^{n}(c)$ with Gaussian curvature $K$. Then $\nabla T=A_{N}$ and
$$
-\frac{1}{2} \Delta \vert T \vert ^{2}= K \vert T \vert ^{2}+ \vert A_{N}\vert ^{2}.
$$
\end{theorem}

\begin{proof}
It is well-known that 
\begin{eqnarray}\label{Laplacian of T}
-\frac{1}{2}\Delta \vert T \vert ^{2} &=& \langle  \trace \nabla^{2}T, T\rangle + \vert \nabla T \vert ^{2}.
\end{eqnarray}

Now, since $M^{2}$ is totally real, taking the inner product of (\ref{Equality}) with $Y$, it follows that 
$$
\langle \nabla _{X}T, Y \rangle = \langle A_{N}X,Y \rangle
$$
for any  vector fields $X$, $Y$ tangent to $M^{2}$, and then
\begin{eqnarray*}
\nabla_{X}T=A_{N}X, 
\end{eqnarray*}
that is $\nabla T = A_{N}$, and therefore 

\begin{eqnarray}\label{norm AN}
\vert \nabla T \vert= \vert A_{N}\vert.
\end{eqnarray}

Now, we compute the first term in the right hand side of Equation (\ref{Laplacian of T}), and prove that 
\begin{eqnarray}\label{trace nabla two}
\langle\trace \nabla^{2}T, T\rangle = K\vert T \vert ^{2}.
\end{eqnarray} 

First, we note that, from the decomposition of $JH$, we get $N \perp H$, and obtain
\begin{eqnarray}\label{trace shape operator}
\trace A_{N}= 
2\langle H, N\rangle
=0.
\end{eqnarray}

Let $\{E_{i}\}_{i=1}^{2}$ be a local (positive) orthonormal frame field geodesic at $p\in M^{2}$. Then, at $p$, we have
\begin{eqnarray}\label{nabla T}
\langle  \trace \nabla^{2}T, T\rangle &=&
\langle \nabla_{E_{i}} \nabla_{E_{i}}T, T\rangle \nonumber
\\
&=& \langle \nabla_{E_{i}}A_{N}E_{i}, T\rangle \nonumber
\\
&=&
\langle   (\nabla_{E_{i}}A_{N})E_{i},T\rangle \nonumber
\\
&=&
\langle   (\nabla_{E_{i}}A_{N})T,E_{i}\rangle,
\end{eqnarray}
where we used the fact that $\langle(\nabla_{X} A_{N})\cdot, \cdot\rangle$ is symmetric.
 
The Codazzi equation \eqref{the Codazzi Equation} becomes
\begin{eqnarray} \label{Codazzi}
(\nabla_{X}^{\perp}B)(Y,Z)-(\nabla_{Y}^{\perp}B)(X,Z)=(\overline{R}(X,Y)Z)^{\perp}=0,
\end{eqnarray}
since $M^{2}$ is totally real.

In (\ref{Codazzi}), we consider $X= E_{i}$, $Y=E_{j}$ and $Z= E_{k}$.  At $p$ we have
\begin{eqnarray*}
\langle ( \nabla_{E_{i}}^{\perp} B)(E_{j},E_{k}), N\rangle &=&
\langle \nabla_{E_{i}}^{\perp}B(E_{j},E_{k}),N\rangle
\\
&=&E_{i}\langle B(E_{j},E_{k}),N \rangle-\langle B(E_{j},E_{k}),\nabla_{E_{i}}^{\perp}N\rangle 
\\
&=&
E_{i}\langle A_{N}E_{j},E_{k}\rangle -\langle B(E_{j},E_{k}),\nabla_{E_{i}}^{\perp}N\rangle.
\end{eqnarray*}
Then using Equation (\ref{Equality}) we obtain
\begin{eqnarray}
\label{Codazzi nabla B}
\langle ( \nabla_{E_{i}}^{\perp} B)(E_{j},E_{k}), N\rangle 
&=&\langle \nabla_{E_{i}}(A_{N}E_{j}),E_{k}\rangle \nonumber 
\\
&\ & -\langle B(E_{j},E_{k}),-J(A_{H}E_{i})-\nabla_{E_{i}}T-B(E_{i},T)+A_{N}E_{i}\rangle \nonumber
\\
&=&
\langle (\nabla_{E_{i}} 
A_{N})E_{j},E_{k}\rangle +\langle B(E_{j},E_{k}),J(A_{H}E_{i})\rangle\nonumber
\\
&\ &
+\langle B(E_{j},E_{k}), B(E_{i},T)\rangle,
\end{eqnarray} 
and similarly
\begin{eqnarray}\label{Codazzi2}
\langle (\nabla_{E_{j}}^{\perp} B)(E_{i},E_{k}), N\rangle &=&\langle (\nabla_{E_{j}}A_{N})E_{i},E_{k}\rangle +\langle B(E_{i},E_{k}),J(A_{H}E_{j})\rangle\nonumber
\\
&\ &
+\langle B(E_{i},E_{k}), B(E_{j},T)\rangle. 
\end{eqnarray}
From Equations (\ref{Codazzi}), (\ref{Codazzi nabla B}) and (\ref{Codazzi2}) we have
\begin{eqnarray}\label{Codazzi3}
\langle (\nabla_{E_{i}}A_{N})E_{j},E_{k}\rangle 
- \langle (\nabla_{E_{j}}A_{N})E_{i},E_{k}\rangle
&=& \langle B(E_{j},T), B(E_{i},E_{k})\rangle \nonumber
\\
& \ &
- \langle B(E_{i},T),B(E_{j},E_{k})\rangle \nonumber
\\
&\ &
+ \langle B(E_{i},E_{k}),J(A_{H}E_{j})\rangle \nonumber
\\
&\ & - \langle B(E_{j},E_{k}),J(A_{H}E_{i})\rangle.
\end{eqnarray}
Further, using the fact that $M^{2}$ is totally real, we have
\begin{eqnarray*}
\langle B(E_{i},E_{k}),J(A_{H}E_{j})\rangle &=& -\langle JB(E_{i},E_{k}),A_{H}E_{j}\rangle
\\
&=&
-\langle J(\overline{\nabla}_{E_{i}}E_{k}),A_{H}E_{j}\rangle
\\ &=&
-\langle \overline{\nabla}_{E_{i}}JE_{k},A_{H}E_{j}\rangle
\\&=&
\langle JE_{k},\overline{\nabla}_{E_{i}}(A_{H}E_{j})\rangle
\\&=&
-\langle JE_{k}, \overline{\nabla}_{E_{i}}\overline{\nabla}_{E_{j}}H \rangle
\\
&=&
\langle E_{k}, J(\overline{\nabla}_{E_{i}}\overline{\nabla}_{E_{j}}H)\rangle,
\end{eqnarray*}
 where we also used $\nabla^{\perp}H=0$. Similarly, it follows that
$$
\langle B(E_{j},E_{k}),J(A_{H}E_{i})\rangle=\langle E_{k},J(\overline{\nabla}_{E_{j}}\overline{\nabla}_{E_{i}}H)\rangle,
$$
and, therefore,
\begin{eqnarray}\label{2nd ff}
\\
\langle B(E_{i},E_{k}),J(A_{H}E_{j})\rangle - \langle B(E_{j},E_{k}), J(A_{H}E_{i})\rangle =\langle E_{k},J(\overline{R}(E_{i},E_{j})H)\rangle. \nonumber
\end{eqnarray} 

Next, from the Gauss Equation \eqref{the Gauss Equation}, we have
\begin{eqnarray}\label{Gauss Equation}
\langle B(E_{i},E_{k}),B(E_{j},T)\rangle - \langle B(E_{j},E_{k}),B(E_{i},T)\rangle 
\\
=\langle \overline{R}(E_{i},E_{j})E_{k},T\rangle-\langle R(E_{i},E_{j})E_{k},T\rangle.\nonumber
\end{eqnarray} 
Now, we compute the first term in the right hand side of (\ref{Gauss Equation}) and the curvature term  in (\ref{2nd ff}). We have
\begin{eqnarray*}
\langle\overline{R}(E_{i},E_{j})E_{k},T\rangle &=&\frac{c}{4}\Big \{\langle E_{j},E_{k}\rangle\langle E_{i},T\rangle-\langle E_{i},E_{k}\rangle\langle E_{j},T\rangle +\langle JE_{j},E_{k}\rangle \langle JE_{i},T \rangle
\\
&\ & \quad -\langle JE_{i},E_{k}\rangle \langle JE_{j},T\rangle
+2\langle JE_{j}, E_{i}\rangle \langle JE_{k},T\rangle \Big\}
\\
&=&
\frac{c}{4}\Big \{\langle E_{j},E_{k}\rangle\langle E_{i},T\rangle-\langle E_{i},E_{k}\rangle\langle E_{j},T\rangle \Big \},
\end{eqnarray*}
and
$$
\overline{R}(E_{i},E_{j})H =\frac{c}{4}\Big \{ \langle JE_{j},H\rangle JE_{i}-\langle JE_{i},H\rangle JE_{j}\Big \},
$$
which shows that
\begin{eqnarray*}
\langle J(\overline{R}(E_{i},E_{j})H),E_{k}\rangle &=&
\frac{c}{4} \Big \{ \langle JE_{i},H\rangle\langle E_{j},E_{k}\rangle -\langle JE_{j}, H\rangle\langle E_{i},E_{k}\rangle \Big \}
\\
&=& 
\frac{c}{4} \Big \{ \langle E_{j},JH\rangle\langle E_{i},E_{k}\rangle -\langle E_{i}, JH\rangle\langle E_{j},E_{k}\rangle \Big \}
\\
&=&
\frac{c}{4} \Big \{ \langle E_{j},T\rangle\langle E_{i},E_{k}\rangle -\langle E_{i}, T\rangle\langle E_{j},E_{k}\rangle \Big \}.
\end{eqnarray*}
We note that 
\begin{eqnarray} \label{Equality curvature}
\langle\overline{R}(E_{i},E_{j})E_{k},T\rangle +\langle J(\overline{R}(E_{i},E_{j})H),E_{k}\rangle =0.
\end{eqnarray}
It easily follows from 
Equations (\ref{Codazzi3}), (\ref{2nd ff}), (\ref{Gauss Equation}) and (\ref{Equality curvature}) that 
\begin{eqnarray*}
\langle (\nabla_{E_{i}}A_{N})E_{j},E_{k}\rangle -\langle (\nabla_{E_{j}}A_{N})E_{i},E_{k}\rangle &=&\langle\overline{R}(E_{i},E_{j})E_{k},T\rangle -\langle R(E_{i},E_{j})E_{k},T\rangle 
\\
& \ &+\langle E_{k},J\overline{R}(E_{i},E_{j})H\rangle.
\\
&=& -\langle R(E_{i},E_{j})E_{k},T\rangle.
\end{eqnarray*}

In the above relation, because of its tensorial character, we can consider $E_{j}=T$. Then taking $k=i$, summing up over $i$ and using (\ref{trace shape operator}), we get  
\begin{eqnarray}\label{Relationship between K and T}
\sum_{i=1}^{2}\langle (\nabla_{E_{i}}A_{N})T,E_{i}\rangle &=& \sum_{i=1}^{2} \Big \{ T\langle A_{N}E_{i},E_{i}\rangle -\langle R(E_{i},T)E_{i},T\rangle \Big \}\nonumber
\\
&=&   T (\trace A_{N}) -\sum_{i=1}^{2}\langle R(E_{i},T)E_{i},T\rangle\nonumber
\\&=&
\vert T \vert ^{2}K.
\end{eqnarray}
Thus, from  Equations  (\ref{nabla T}) and (\ref{Relationship between K and T}), we obtain Equation (\ref{trace nabla two}).

Finally, we use (\ref{norm AN}) and (\ref{trace nabla two}) to conclude.
\end{proof}

\begin{theorem}\label{theorem on complete PMC}
If $M^{2}$ is a complete PMC totally real surface with $K\geq 0$ in a complex space form $N^{n}(c)$. Then $\nabla T=A_{N}=0$, and either 
$K=0$ or $K>0$ at some point and $T=0$.
\end{theorem}

\begin{proof}
As $|T|^2\leq \vert JH \vert ^{2}= \vert H \vert ^{2}$ and $M^{2}$ is CMC, we have that $\vert T \vert ^{2}$ is a bounded function on $M^{2}$. Further, since $\Delta|T|^2\leq 0$, i.e. $|T|^2$ is  a subharmonic function, it follows that $|T|^2$ is constant (see \cite{Huber}).
Thus,
$K\vert T \vert ^{2}+\vert A_{N}\vert ^{2}=0$, 
and so
$\vert A_{N}\vert ^{2}=0$ and $K\vert T \vert ^{2}=0$.
Therefore $A_{N}= \nabla T=0$, and either $K=0$ (everywhere) or $K>0$ at some point and $T=0.$

\end{proof}

From Corollary \ref{Corollary on complete PMC} and Theorem \ref{theorem on complete PMC} we have the following result.

\begin{corollary}
Let $M^{2}$ be a complete PMC totally real surface  with $K\geq 0$ in a complex space form $N^{n}(c)$. Then 
$$
\nabla A_{H}=\nabla T=A_{N}=0,
$$
and either $M^{2}$ is flat or it is pseudo-umbilical with $T=0$.
\end{corollary}

\section{CMC biconservative surfaces in $N^{2}(c)$}

Consider $M^{2}$ a CMC surface in a complex space form $N^{2}(c)$ of complex dimension $2$, with $c\neq 0$. Let $\{E_{3}=H/\vert H \vert,  E_{4}\}$ be the global orthonormal frame field in the normal bundle $NM^{2}$, and $\{E_{1}, E_{2} \}$ a local positive orthonormal frame field tangent to $M^{2}$. Then the frame field 
$$
\{E_{1}, E_{2},E_{3},E_{4}\}
$$ 
along $M^{2}$ can be extended to a local orthonormal frame field defined on an open subset of $N^{2}(c)$ and tangent to $N^{2}(c)$.

Denote by $\omega_{A}^{B}$ be the connection $1$-forms corresponding to $
\{E_{1}, E_{2},E_{3},E_{4}\}
$, i.e. on $N^{2}(c)$ we have 
$$
\overline{\nabla}_{\cdot}E_{A}=\omega_{A}^{B}(\cdot)E_{B},
$$
and by $\{ \omega^{1}, \omega^{2}, \omega^{3}, \omega^{4}\}$ the dual basis of $
\{E_{1}, E_{2},E_{3},E_{4}\}
$.
It follows that on $M^{2}$ the following relations hold
\begin{eqnarray}\label{Connections}
\nabla^{\perp}_{E_{1}}E_{3}=\omega_{3}^{4}(E_{1})E_{4}, \quad \nabla^{\perp}_{E_{2}}E_{3}= \omega_{3}^{4}(E_{2})E_{4}.
\end{eqnarray}

\begin{proposition}\label{PMC prop1}
Let $M^{2}$ be a pseudo-umbilical CMC biconservative  surface in a complex space form $N^{2}(c)$, with $c\neq 0$. Then $M^{2}$ is PMC and JT is normal.
\end{proposition}

\begin{proof}
From the definition of the curvature tensor field and from the fact that $M^{2}$ is pseudo-umbilical, we get  
\begin{eqnarray}\label{tensorial curvature of N}
\overline{R}(X,Y)H&=& \overline{\nabla}_{X}\overline{\nabla}_{Y}H-\overline{\nabla}_{Y}\overline{\nabla}_{X}H-\overline{\nabla}_{[X,Y]}H
\nonumber
\\
&=&
\overline{\nabla}_{X}(\nabla_{Y}^{\perp}H-A_{H}Y)-\overline{\nabla}_{Y}(\nabla_{X}^{\perp}H-A_{H}X)-\nabla^{\perp}_{[X,Y]}H +A_{H}[X,Y]
\nonumber
\\
&=&
\overline{\nabla}_{X}\nabla_{Y}^{\perp}H-\overline{\nabla}_{X}(\vert H \vert ^{2}Y)-\overline{\nabla}_{Y}\nabla_{X}^{\perp}H+ \overline{\nabla}_{Y}(\vert H \vert ^{2}X)-\nabla^{\perp}_{[X,Y]}H 
\nonumber
\\
&\ &
+\vert H \vert ^{2}[X,Y]\nonumber
\\
&=&
\vert H \vert ^{2}(\overline{\nabla}_{Y}X-\overline{\nabla}_{X}Y+[X,Y])+\overline{\nabla}_{X}\nabla_{Y}^{\perp}H-\overline{\nabla}_{Y}\nabla_{X}^{\perp}H-\nabla^{\perp}_{[X,Y]}H
\nonumber
\\
&=&
\nabla_{X}^{\perp}\nabla_{Y}^{\perp}H-A_{\nabla_{Y}^{\perp}H}X-\nabla^{\perp}_{Y}\nabla_{X}^{\perp}H+A_{\nabla_{X}^{\perp}H}Y-\nabla^{\perp}_{[X,Y]}H
\nonumber
\\
&=&
A_{\nabla_{X}^{\perp}H}Y-
A_{\nabla_{Y}^{\perp}H}X+R^{\perp}(X,Y)H,
\end{eqnarray}
for any $X$, $Y$ tangent to $M^{2}$.

Now, from the Ricci Equation \eqref{the Ricci Equation} and as $M^{2}$ is pseudo-umbilical, we obtain
\begin{eqnarray*}
\langle R^{\perp}(X,Y)H,V\rangle &=& \langle [A_{H},A_{V}]X,Y\rangle +\langle \overline{R}(X,Y)H,V\rangle
\\
&=&
\langle \overline{R}(X,Y)H,V \rangle,
\end{eqnarray*}
and, from (\ref{tensorial curvature of N}), we get 
\begin{eqnarray}\label{Shape operators}
A_{\nabla_{X}^{\perp}H}Y=A_{\nabla^{\perp}_{Y}H}X,
\end{eqnarray}
for any $X$ and $Y$ tangent to $M^{2}$.

From (\ref{Connections}) and (\ref{Shape operators}) for $X=E_{1}$ and $Y=E_{2}$, we obtain
\begin{eqnarray}\label{1-form connections} 
\omega_{3}^{4}(E_{1})A_{4}E_{2}=\omega_{3}^{4}(E_{2})A_{4}E_{1},
\end{eqnarray}
where $A_{i}=A_{E_{i}}$, $i\in\{3,4\}$.

Assume that $\nabla^{\perp}H\neq 0$. Then there exists an open subset of $M^{2}$ where 
$\nabla^{\perp}H \neq 0$ at any point, i.e. $\omega_{3}^{4}\neq 0$ at any point, and we will work on that subset. For the sake of simplicity, we can assume that this subset is the whole manifold $M^{2}$.

Now, let 
$$
A_{4}=
\begin{bmatrix} 
\mu_{1} & \mu_{0}  \\
\mu_{0} & \mu_{2}
\end{bmatrix}, 
$$
with respect to $\{E_{1},E_{2}\}$. Since $\trace A_{4}=2\langle H,E_{4}\rangle=0,$ we obtain 
$\mu_{2}=-\mu_{1}$, and therefore
$$
A_{4}=
\begin{bmatrix} 
\mu_{1} & \mu_{0}  \\
\mu_{0} & -\mu_{1}
\end{bmatrix}.
$$

Now, using Equation(\ref{1-form connections}), we obtain 
$$
\omega_{3}^{4}(E_{1})\{ \mu_{0}E_{1}-\mu_{1}E_{2}\}=\omega_{3}^{4}(E_{2})\{ \mu_{1}E_{1}+\mu_{0}E_{2}\},
$$
that is 
\begin{eqnarray*}
\left\lbrace
\begin{array}{ccc}
\omega_{3}^{4}(E_{1})\mu_{0}-\omega_{3}^{4}(E_{2})\mu_{1}&=&0
\\
-\omega_{3}^{4}(E_{1})\mu_{1}-\omega_{3}^{4}(E_{2})\mu_{0}&=&0.
\end{array}\right.
\end{eqnarray*}
It follows that
$$
\{(\omega_{3}^{4}(E_{2}))^{2}+(\omega_{3}^{4}(E_{1}))^{2}\}\mu_{1}=0
$$
and, since $\nabla^{\perp}H\neq0$, we obtain $\mu_{1}=\mu_{2}=0$, and then $\mu_{0}=0$.
Thus, $A_{4}=0$ on $M^{2}$.

Now, we have
$$
B(E_{1},E_{1})=\vert H \vert E_{3}, \quad  B(E_{2},E_{2})=\vert H \vert E_{3} \quad  \textnormal{and} \quad B(E_{1},E_{2})=0.
$$

Since $M^{2}$ is CMC and biconservative, we have 
$$
\trace A_{\nabla_{\cdot}^{\perp}H}(\cdot)+\trace(\overline{R}(\cdot,H)\cdot)^{\top}=0.
$$
Also, using $A_{4}=0$, we get
$$
\trace A_{\nabla_{\cdot}^{\perp}H}(\cdot)= \vert H \vert \omega_{3}^{4}(E_{1})A_{4}E_{1}+\vert H \vert \omega_{3}^{4}(E_{2})A_{4}E_{2}=0.
$$
Therefore, 
$\trace(\overline{R}(\cdot,H)\cdot)^{\top}=0$, which shows that, as $c\neq 0$, $(JT)^{\top}=0$.

Next, we will use again the Codazzi Equation. From  Equation (\ref{Tensorial Curvature}) we have
\begin{eqnarray*}
(\overline{R}(X,Y)Z)^{\perp}&=&\frac{c}{4} \{\langle JY,Z\rangle (JX)^{\perp}-\langle JX,Z\rangle (JY)^{\perp}+2\langle JY,X\rangle(JZ)^{\perp} \},
\end{eqnarray*}
and if 
\\
(i) $X=Z=E_{1}$, $Y=E_{2}$, we obtain
\begin{eqnarray*}
\langle (\overline{R}(E_{1},E_{2})E_{1})^{\perp},E_{3}\rangle &=& \frac{3c}{4}\langle JE_{2},E_{1}\rangle\langle JE_{1},E_{3}\rangle
=
\frac{3c}{4\vert H \vert}\langle JE_{2},E_{1}\rangle\langle JE_{1},H\rangle
\\
&=&
-\frac{3c}{4\vert H \vert}\langle JE_{2},E_{1}\rangle\langle E_{1},JH\rangle
=
-\frac{3c}{4\vert H \vert}\langle JE_{2},E_{1}\rangle\langle E_{1},T\rangle
\\
&=&
-\frac{3c}{4\vert H \vert}\langle JE_{2},T\rangle
=
-\frac{3c}{4\vert H \vert}\langle E_{2},JT\rangle
\\
&=&
0
\end{eqnarray*}
since $JT$ is normal. In the same way we obtain
\begin{eqnarray}\label{1-1Codazzi}
\langle (\overline{R}(X,Y)Z)^{\perp},V\rangle=0
\end{eqnarray}
in all the following cases:
\\
(ii) $X=E_{1}$, $Y=Z=E_{2}$, $V=E_{3},$
\\
(iii) $X=Z=E_{1}$, $Y=E_{2}$, $V=E_{4},$
\\
(iv) $X=E_{1}$, $Y=Z=E_{2}$, $V=E_{4}$.

Hence, for $X$, $Y$, $Z$ and $V$ as in any of the above cases, from the Codazzi Equation \eqref{the Codazzi Equation} and from Equation (\ref{1-1Codazzi}), we get
$$
\vert H \vert\omega_{3}^{4}(E_{2})=0 \quad \textnormal{and} \quad 
\vert H \vert \omega_{3}^{4}(E_{1})=0,
$$
that is $\omega_{3}^{4}= 0$, which is a contradiction.
\end{proof}

\begin{remark}
Taking into account Theorem \ref{Theorem}, Proposition \ref{PMC prop1} agrees with the result in \cite{N.Sato}.
\end{remark}

\begin{proposition}\label{PMC prop2}
Let $M^{2}$ be a CMC biconservative surface with no pseudo-umbilical points in a complex space form $N^{2}(c)$, with $c\neq 0$. If $JT$ is normal, then $M^{2}$ is PMC.
\end{proposition}

\begin{proof}
Let $\{ \lambda_{1}, \lambda_{2}\}$ be the  smooth eigenvalue functions of $A_{3}$ on $M^{2}$ and one can consider $\{E_{1}, E_{2}\}$ such that
$$
\lambda_{1} < \lambda_{2}, \quad A_{3}E_{1}=\lambda_{1}E_{1} \quad  \textnormal{and} \quad  A_{3}E_{2}=\lambda_{2}E_{2}.
$$
We note that $\trace A_{4}=0$.

Assume that 
$\nabla^{\perp}H\neq 0$. Then there exists an open subset of $M^{2}$ such that $\nabla^{\perp}H\neq 0$ at any point of this subset, and from now on we will work only here. For the sake of simplicity, we assume that this open subset is the whole manifold $M^{2}$.

As $M^{2}$ is CMC and $JT$ is normal, the biconservative condition
$$
2\trace A_{\nabla_{\cdot}^{\perp}H}(\cdot)+\grad (\vert H \vert ^{2})+2\trace(\overline{R}(\cdot,H)\cdot)^{\top}=0
$$
reduces to 
\begin{eqnarray}\label{traceA_nabla}
\trace A_{\nabla_{\cdot}^{\perp}H}(\cdot)=0.
\end{eqnarray}

Since
\begin{eqnarray*}
\trace A_{\nabla_{\cdot}^{\perp}E_{3}}(\cdot)&=&A_{\nabla_{E_{1}}^{\perp}E_{3}}E_{1}+A_{\nabla_{E_{2}}^{\perp}E_{3}}E_{2}
\\
&=&
A_{\omega_{3}^{4}(E_{1})E_{4}}E_{1}+A_{\omega_{3}^{4}(E_{2})E_{4}}E_{2}
\\
&=&
\omega_{3}^{4}(E_{1})A_{4}E_{1}+\omega_{3}^{4}(E_{2})A_{4}E_{2},
\end{eqnarray*}
the Equation (\ref{traceA_nabla}) can be written as  
\begin{eqnarray}\label{system}
\left\lbrace
\begin{array}{ccc}
\omega_{3}^{4}(E_{1})\langle A_{4}E_{1}, E_{1}\rangle +\omega_{3}^{4}(E_{2})\langle A_{4}E_{2}, E_{1}\rangle
&=&0
\\
\omega_{3}^{4}(E_{1})\langle A_{4}E_{1}, E_{2}\rangle +\omega_{3}^{4}(E_{2})\langle A_{4}E_{2}, E_{2}\rangle&=&0.
\end{array}\right.
\end{eqnarray}
Since $\nabla^{\perp} H\neq 0$, i.e.
\begin{eqnarray*}
\vert \nabla^{\perp}E_{3}\vert ^{2}=(\omega _{3}^{4}(E_{1}))^{2}+(\omega _{3}^{4}(E_{2}))^{2}>0,
\end{eqnarray*}
we have that the system (\ref{system}) has a non-trivial solution. Therefore, its determinant is zero, i.e. 
\begin{eqnarray}\label{norm of A4}
0&=&\langle A_{4}E_{1},E_{1} \rangle \langle A_{4}E_{2},E_{2}\rangle - (\langle A_{4}E_{1},E_{2}\rangle)^{2}\nonumber
\\
&=&
-(\langle A_{4}E_{1},E_{1} \rangle)^{2}-(\langle A_{4}E_{1},E_{2} \rangle)^{2}\nonumber
\\
&=&
-\vert A_{4} E_{1}\vert^{2},
\end{eqnarray}
where in the second equality we  used $\trace A_{4}=0$. Moreover, since 
$$
\vert A_{4}E_{1}\vert ^{2}=\vert A_{4}E_{2}\vert ^{2},
$$
from (\ref{norm of A4}) we get
$$
\vert A_{4} \vert ^{2}=2\vert A_{4}E_{1} \vert ^{2}=0, 
$$
that is $A_{4}=0$ on $M^{2}$.

The second fundamental form $B$ of $M^{2}$ is given by 
$$
B(E_{1},E_{1})=\lambda_{1}E_{3}, \quad B(E_{2},E_{2})=\lambda_{2}E_{3} \quad \textnormal{and} \quad B(E_{1},E_{2})=0,
$$
and therefore, $2H=(\lambda_{1}+\lambda_{2})E_{3}$ and $\lambda_{1}+\lambda_{2}=2\vert H \vert=constant \neq 0$.

Next, we will use again the Codazzi Equation. From  Equation (\ref{Tensorial Curvature}) we have
\begin{eqnarray*}
(\overline{R}(X,Y)Z)^{\perp}&=&\frac{c}{4} \{\langle JY,Z\rangle (JX)^{\perp}-\langle JX,Z\rangle (JY)^{\perp}+2\langle JY,X\rangle(JZ)^{\perp} \},
\end{eqnarray*}
and if 
\\
(i) $X=Z=E_{1}$, $Y=E_{2}$, we obtain
\begin{eqnarray*}
\langle (\overline{R}(E_{1},E_{2})E_{1})^{\perp},E_{3}\rangle &=& \frac{3c}{4}\langle JE_{2},E_{1}\rangle\langle JE_{1},E_{3}\rangle
=
\frac{3c}{4\vert H \vert}\langle JE_{2},E_{1}\rangle\langle JE_{1},H\rangle
\\
&=&
-\frac{3c}{4\vert H \vert}\langle JE_{2},E_{1}\rangle\langle E_{1},JH\rangle
=
-\frac{3c}{4\vert H \vert}\langle JE_{2},E_{1}\rangle\langle E_{1},T\rangle
\\
&=&
-\frac{3c}{4\vert H \vert}\langle JE_{2},T\rangle
=
-\frac{3c}{4\vert H \vert}\langle E_{2},JT\rangle
\\
&=&
0
\end{eqnarray*}
since $JT$ is normal.
In the same way we obtain
\begin{eqnarray}\label{1-codazzi2}
\langle (\overline{R}(X,Y)Z)^{\perp},V\rangle=0
\end{eqnarray}
for
\\
(ii) $X=E_{1}$, $Y=Z=E_{2}$, $V=E_{3},$
\\
(iii) $X=Z=E_{1}$, $Y=E_{2}$, $V=E_{4},$
\\
(iv) $X=E_{1}$, $Y=Z=E_{2}$, $V=E_{4}$.

Hence, for $X$, $Y$, $Z$ and $V$ as in the previous cases, the Codazzi Equation  (\ref{the Codazzi Equation}) in each case is as follows:
\begin{enumerate}
\item \label{derivation of E1}
$E_{2}(\lambda_{1})=(\lambda_{1}-\lambda_{2})\omega_{1}^{2}(E_{1})$,
\item\label{derivative of E2}
$E_{1}(\lambda_{2})=(\lambda_{1}-\lambda_{2})\omega_{1}^{2}(E_{2})$,
\item\label{product lambda1}
$\lambda_{1}\omega_{3}^{4}(E_{2})=0,$
\item \label{product lambda2}
$\lambda_{2}\omega_{3}^{4}(E_{1})=0.$
\end{enumerate}

Assume that $\omega_{3}^{4}(E_{1})\neq 0$ at some point $p\in M^{2}$, then 
$\omega_{3}^{4}(E_{1})\neq 0$ on a neighborhood of 
$p$. On this neighborhood, by (\ref{product lambda2}) we have $\lambda_{2}=0$ and so 
$\lambda_{1}=2\vert H \vert$, which is a contradiction since 
$\lambda_{1}<\lambda_{2}=0$.

If $\omega_{3}^{4}(E_{2})\neq0$, we get $\lambda_{1}=0$ on an open subset and so $\lambda_{2}=2\vert H \vert$. From (\ref{derivation of E1}) and (\ref{derivative of E2}), we obtain $\omega_{1}^{2}(E_{1})=0$ and $\omega_{1}^{2}(E_{2})=0$. We will use the same notation 
$\omega_{1}^{2}$ for the pull-back of $\omega_{1}^{2}$ on $M^{2}$. Therefore, on $M^{2}$, 
$\omega_{1}^{2}=0$ and 
$\nabla_{E_{i}}E_{j}=0$, for any $i,j\in\{1,2\}$. Since the curvature of $M^{2}$ is given by $d\omega_{1}^{2}=-K\omega^{1}\wedge \omega^{2}$, we conclude that $M^{2}$ is flat.

Further, from the Gauss Equation \eqref{the Gauss Equation}, for $X=W=E_{1}$, $Y=Z=E_{2}$ and using the fact that $M^{2}$ is flat, we obtain
$$
\frac{c}{4}\{1+3\langle JE_{1},E_{2}\rangle ^{2} \}=-\lambda_{1} \lambda_{2}=0,
$$
which is a contradiction, as $c\neq 0$.

Therefore $\nabla^{\perp}H=0$.

Even if we got a contradiction and proved that $H$ is parallel, we note that, when $H$ is not parallel, $\lambda_{1}=0$, 
 $\lambda_{2}=2\vert H \vert$, $\omega_{1}^{2}=0$, $K=0$ and $\omega_{3}^{4}(E_{1})=0$, from the Ricci Equation \eqref{the Ricci Equation}, we obtain
$$
E_{1}(\omega_{3}^{4}(E_{2}))=\frac{c}{4\vert H \vert}\Big{\{} -\langle E_{2},T\rangle \langle JE_{1},E_{4}\rangle+\langle E_{1},T\rangle \langle JE_{2},E_{4}\rangle+2\langle JE_{2},E_{1}\rangle\langle N,E_{4}\rangle \Big{\}}.
$$
\end{proof}

From Propositions \ref{PMC prop1} and \ref{PMC prop2}, we get the following theorem.

\begin{theorem}\label{thm to be generalised}
Let $M^{2}$ be a CMC biconservative surface in a complex space form $N^{2}(c)$, with $c\neq 0$. If $JT$ is normal, then $M^{2}$ is PMC.
\end{theorem}

\begin{proof}
We know that, with standard notations, 
$$
\langle A_{H}\partial_{z},\partial_{z}\rangle 
$$
is holomorphic (see \cite{Nistor1}). Therefore, either $M^{2}$ is pseudo-umbilical in the complex space form $N^{2}(c)$, or the set $W$ of the non pseudo-umbilical points is an open and dense subset of $M^{2}$.

In the first case, the result follows directly from Proposition \ref{PMC prop1}. In the second case, from Proposition \ref{PMC prop2} we get that $W$ is PMC in the complex space form $N^{2}(c)$, and then by continuity we obtain $M^{2}$ is PMC.
\end{proof}

\begin{remark}
When the ambient space is a real space form of dimension $4$, a similar result to Theorem \ref{thm to be generalised} was obtained in   \cite[Theorem 5.1]{M.O.R.}.
\end{remark}

In the following, we want to check if one can extend the above result to the case $c=0$, i.e. we want to see whether the CMC biconservative surfaces in  $\mathbb{C}^{2}=\mathbb{E}^{4}$ with $JT$ normal are PMC. Equivalently, we investigate if the CMC biconservative surfaces which are not PMC have 
$(JT)^{\top}\neq 0$.

The parametric equations for the CMC biconservative surfaces which are not PMC were given in \cite{M.O.R.}.

\begin{proposition}(\cite{M.O.R.})
Let $M^{2}$ be a non-PMC biconservative surface with constant mean curvature in
$\mathbb{E}^{4}$. Then, locally, the surface is given by 
\begin{eqnarray}\label{surface}
X(u,v)=(\gamma^{1}(u),\gamma^{2}(u),\gamma^{3}(u), v),
\end{eqnarray} 
where $\gamma :I \to \mathbb{E}^{3}$ is a curve in $\mathbb{E}^{3}$ parametrized by arc-length, with constant non-zero curvature, and non-zero torsion.
\end{proposition}

Now we will prove that any CMC biconservative surface which is not PMC  has $(JT)^{\top}\neq 0$.

\begin{proposition}
Let $M^{2}$  be a non-PMC biconservative surface with constant mean curvature in 
$\mathbb{E}^{4}$. Then $(JT)^{\top}\neq 0$.
\end{proposition}

\begin{proof}
Let $X(u,v)=(\gamma^{1}(u),\gamma^{2}(u),\gamma^{3}(u), v)$, where $\gamma :I \to \mathbb{E}^{3}\equiv \mathbb{E}^{3}\times\{ 0\}\subset \mathbb{E}^{4}$ is a curve parametrized by arc-length, i.e. $\vert \gamma ^{'} \vert =1$, with $\kappa=constant$, $\kappa \neq 0$, and $\tau \in C^{\infty}(I)$ is a non-zero function (we can assume that $\tau > 0$). We denote the Frenet frame field along $\gamma$ by 
$$
\{\gamma^{'}(u), {\bf n}(u), {\bf b}(u)\},\quad u\in I.
$$
We have
\begin{eqnarray*}
\left\lbrace
\begin{array}{ccc}
X_{u}&=& (\gamma^{'},0)=\gamma^{'}\\\
X_{v}&=& e_{4} .
\end{array}\right.
\end{eqnarray*}
It is clear that ${\bf n}$ and ${\bf b}$ are orthogonal to $\gamma^{'}$ and $e_{4}$, thus $\{{\bf n}, {\bf b}\}$ is an orthonormal frame field in the normal bundle of $M^{2}$ in $\mathbb{E}^{4}$. 

Also, we have
$g_{11}=\vert X_{u}\vert ^{2}=1$, $g_{12}=\langle X_{u},X_{v}\rangle=0$ and $g_{22}=\vert X_{v} \vert ^{2}=1$.
\begin{eqnarray*}
\left\lbrace
\begin{array}{ccc}
\nabla_{X_{u}}^{\mathbb{R}^{4}}X_{u}&=& X_{uu}=\gamma^{''}=\kappa {\bf n}
\\\
\nabla_{X_{u}}^{\mathbb{R}^{4}}X_{v}&=& \nabla_{X_{v}}^{\mathbb{R}^{4}}X_{u}=0
\\\ 
\nabla_{X_{v}}^{\mathbb{R}^{4}}X_{v}&=& X_{vv}=0,
\end{array}\right.
\end{eqnarray*}
then $B(X_{u},X_{u})=\kappa {\bf n}$, $B(X_{u},X_{v})=0$ and $B(X_{v},X_{v})=0$.
We have 
\begin{eqnarray*}
H&=&\frac{1}{2}\trace B
= \frac{1}{2} \{B(X_{u},X_{u})+B(X_{v},X_{v})\}
\\
&=&\frac{1}{2}\kappa {\bf n}=\frac{1}{2} \gamma^{''},
\end{eqnarray*}
and, therefore, 
$$
JH=\frac{1}{2}J\gamma^{''}.
$$

Now, the tangential part of $JH$ is given by
\begin{eqnarray*}
T=(JH)^{\top}&=&\langle JH,X_{u}\rangle X_{u}+\langle JH,X_{v}\rangle X_{v}
\\
&=&
\frac{1}{2}\{\langle J\gamma^{''},\gamma^{'}\rangle \gamma^{'}+\langle J\gamma^{''}, e_{4}\rangle e_{4}\}.
\end{eqnarray*}

We will prove that $JT$ is not normal. For this purpose, we will assume that $JT$ is normal and come to a contradiction. 

Since $JT$ is normal, we have 
$$
\langle JT,X_{u}\rangle=0 \quad \textnormal{and} \quad \langle JT,X_{v}\rangle=0.
$$
Thus
\begin{eqnarray*}
\langle JT,X_{u}\rangle =0 
&\Leftrightarrow &
\langle J\gamma^{''},\gamma^{'}\rangle \langle J\gamma^{'},\gamma^{'}\rangle + \langle J\gamma^{''},e_{4}\rangle \langle Je_{4},\gamma^{'}\rangle
=0
\\
&\Leftrightarrow &\langle J\gamma^{''},e_{4}\rangle \langle Je_{4},\gamma^{'}\rangle
=0,
\end{eqnarray*}
and 
\begin{eqnarray*}
\langle JT,X_{v}\rangle =0 
&\Leftrightarrow &
\langle J\gamma^{''},\gamma^{'}\rangle \langle J\gamma^{'},e_{4}\rangle + \langle J\gamma^{''},e_{4}\rangle \langle Je_{4},e_{4}\rangle
=0
\\
&\Leftrightarrow &
\langle J\gamma^{''},\gamma^{'}\rangle \langle J\gamma^{'},e_{4}\rangle=0.
\end{eqnarray*}
Therefore, $JT$ is normal if and only if 
$$
\langle J\gamma^{'},e_{4}\rangle =0
$$
on $I$, or
$$
\langle J\gamma^{''},e_{4}\rangle = 0   \quad
\textnormal{and} \quad
\langle J\gamma^{''},\gamma^{'}\rangle = 0.
$$
We have, $J\gamma^{'}=(-(\gamma^{2})^{'},(\gamma^{1})^{'},0,(\gamma^{3})^{'})$ and $J\gamma^{''}=(-(\gamma^{2})^{''},(\gamma^{1})^{''},0,(\gamma^{3})^{''})$, so
$$
\langle J\gamma^{''},e_{4}\rangle = 0 \Leftrightarrow (\gamma^{3})^{''}=0,
$$
$$
\langle J\gamma^{'},e_{4}\rangle =0 \Leftrightarrow (\gamma^{3})^{'}=0
$$
and 
$$
\langle J\gamma^{''},\gamma^{'}\rangle = 0 \Leftrightarrow -(\gamma^{1})^{'}(\gamma^{2})^{''}+(\gamma^{2})^{'}(\gamma^{1})^{''}=0.
$$
Assume that $\langle J\gamma^{'}, e_{4}\rangle =0$, i.e. 
$(\gamma^{3})^{'}=0$. We obtain that $\gamma^{3}=constant$, so $\gamma$ is a plane curve, and therefore $\tau=0$ which is a contradiction.
Hence we are left with the second case, i.e.
$$
(\gamma^{3})^{''}=0 \quad \text{and} \quad -(\gamma^{1})^{'}(\gamma^{2})^{''}+(\gamma^{2})^{'}(\gamma^{1})^{''}=0.
$$
From $(\gamma^{3})^{''}=0$, we get $\gamma^{3}(u)=au+b$, for any $u\in I$ (or a smaller open interval), and from $\vert \gamma^{'}\vert =1$ we obtain 
$((\gamma^{1})^{'})^{2} +((\gamma^{2})^{'})^{2}+a^{2}=1$.
As $\tau \neq 0$, $a^{2}\in (0,1)$.
Then, there exists a smooth function $f$ such that 
$$
(\gamma^{1})^{'}=\sqrt{1-a^{2}} \cos f \quad \text{and} \quad (\gamma^{2})^{'}=\sqrt{1-a^{2}} \sin f
$$ and 
\begin{eqnarray*}
\left\lbrace
\begin{array}{ccc}
(\gamma^{1})^{''}&=&-\sqrt{1-a^{2}} f^{'}\sin f 
\\\
(\gamma^{2})^{''}&=&\sqrt{1-a^{2}}f^{'} \cos f.
\end{array}\right.
\end{eqnarray*}
Then, the condition 
$$
-(\gamma^{1})^{'}(\gamma^{2})^{''}+(\gamma^{2})^{'}(\gamma^{1})^{''}=0 \nonumber
$$
is equivalent to 
$$
-(1-a^{2})f^{'} \cos ^{2}f-(1-a^{2})f^{'} \sin ^{2}f =0
$$
which means
\begin{eqnarray} \label{property 1}
f=constant.
\end{eqnarray} 

On the other hand, 
$$
\kappa^{2} =\vert \gamma^{''}\vert ^{2}=
[(\gamma^{1})^{''}]^{2}+[(\gamma^{2})^{''}]^{2}=(1-a^{2})(f^{'})^{2}.
$$
As $\kappa > 0$ we obtain 
\begin{eqnarray}\label{property 2}
(f^{'})^{2}>0.
\end{eqnarray}

From (\ref{property 1}) and (\ref{property 2}) we get a contradiction.
\end{proof}

In conclusion, we can extend Theorem \ref{thm to be generalised}  to the case $c=0$, and state the following theorem.

\begin{theorem}
Let $M^{2}$ be a CMC biconservative surface in a complex space form $N^{2}(c)$. If $JT$ is normal, then $M^{2}$ is PMC and totally real.
\end{theorem}

\begin{remark}
The full classification of complete PMC surfaces in a complex space form $N^{2}(c)$ was achieved in \cite{K-K1, K-K2, K-K3} and they are totally real flat tori, when $c>0$.
\end{remark}

\section{Reduction of codimension for biconservative surfaces in $N^{n}(c)$}

We  recall that if $M^{2}$ is a PMC surface in a  real Euclidean space $\mathbb{E}^{n}$, then it is biconservative, and it is either pseudo-umbilical (and lies as a minimal surface in a Euclidean hypersphere of $\mathbb{E}^{n}$), or it lies as a CMC (including minimal) surface in a $3$-dimensional sphere (and this sphere lies in $\mathbb{E}^4$), or it lies as a CMC surface in $\mathbb{E}^{3}$ (see  \cite{Chen surfaces with PMC, Chen PMC surfaces, Yau}).

In this section we will assume that $M^{2}$ is a PMC totally real surface in a complex space form $N^{n}(c)$ of complex dimension $n$, $n$ large enough, $c\neq 0$, with $H\neq 0$, and we will get a reduction of codimension result. More precisely, we will reduce the complex dimension of the ambient space for non pseudo-umbilical such surfaces, to $4$ (see Theorem \ref{Reduction theorem}). For this purpose, we will follow closely \cite{dorel}, where it was proved that a non pseudo-umbilical PMC surface in a complex space form $N^{n}(c)$, $c\neq 0$, lies in $N^{5}(c)$. The strategy for obtaining reduction results was initiated and developed in \cite{Alencar, Eschenburg-Tribuzy, Ferreira-Tribuzy}.
Our result is less restricted than that obtained in \cite{Fetcu-Pinheiro}, where under the stronger condition of biharmonicity, the reduction was done to $N^{2}(c)$. We mention that the reduction of codimension 
for totally real submanifolds of complex space forms, with parallel f-structure in the normal bundle, was obtained in \cite{k-N}.

\begin{lemma}
For any vector field $V$ normal to $M^{2}$ and orthogonal to 
$JTM^{2}$, we have 
$[A_{H},A_{V}] = 0$, i.e., $A_{H}$ commutes with $A_{V}$.
\end{lemma}

\begin{proof}
From the Ricci Equation \eqref{the Ricci Equation}, since 
$M^{2}$ is a PMC surface, we have $$
\langle R^{\perp}(X,Y)H,V\rangle=0,
$$
where $X$ and $Y$ are tangent to $M^{2}$.  Also, as $M^{2}$ is totally real and $V$ is orthogonal to $J(TM^{2})$, we obtain
\begin{eqnarray*}
\langle\overline{R}(X,Y)H,V\rangle &=&\frac{c}{4}\Big \{\langle Y,H\rangle\langle X,V\rangle-\langle X,H\rangle\langle Y,V\rangle +\langle JY,H\rangle \langle JX,V\rangle 
\\
&\ &
\quad -\langle JX,H\rangle\langle 
JY,V\rangle +2\langle JY,X \rangle \langle JH,V\rangle \Big \}
\\
&=&0,
\end{eqnarray*}
therefore 
$[A_{H},A_{V}]=0$.
\end{proof}

\begin{corollary}\label{Normalizing}
At any point $p\in M^{2}$,  either H is an umbilical direction, or there exists an orthonormal frame field $\{E_{1},E_{2}\}$ around $p$ that diagonalizes simultaneously $A_{H}$ and $A_{V}$, for any vector field $V$ normal to $M^{2}$ and orthogonal to $JTM^{2}$.  
\end{corollary}

\begin{proposition}\label{Reduction of codimension for biconservative surfaces}
Assume that $H$ is nowhere an umbilical direction. Then there exists a parallel subbundle $L$ of the normal bundle that contains the image of the
second fundamental form $B$ and has real dimension less or equal to $6$.
\end{proposition}

\begin{proof}
We define the subbundle $L$ of the normal bundle of $M^{2}$ in the complex space form $N^{n}(c)$ by 
$$
L=\Span\{\im B\cup (J\im B)^{\perp}\cup JTM^{2}\},
$$
where $(J\im B)^{\perp}=\{ (JB(X,Y))^{\perp}:X$, $Y$ tangent vector fields to $M^{2}\}$.

To prove that $L$ is parallel, let $U$ be a section in $L$, and it is sufficient to show that 
$\nabla_{X}^{\perp}U$ is also a section in $L$, for any $X$. This  means that $\langle \nabla_{X}^{\perp}U,V\rangle=0$, for any $V$ normal to $M^{2}$ and orthogonal to $L$; equivalently, $\langle U, \nabla_{X}^{\perp}V\rangle =0$.

Let $V$ be a normal vector field orthogonal to $L$. This means that 
$$
\langle V,B(X,Y)\rangle =\langle V,JB(X,Y)\rangle =\langle V,JX\rangle 
=0,
$$
for any $X$, $Y$ tangent to $M^{2}$.

Consider $\{E_{1},E_{2}\}$ a local orthonormal frame field that diagonalizes simultaneously $A_{H}$ and $A_{V}$ (see Corollary \ref{Normalizing}). We want to prove $\nabla_{E_{k}}^{\perp}V$ is orthogonal to $\im B$, $(J\im B)^{\perp}$, and $JTM^{2}$.

In order to prove this, we first prove that $\nabla_{E_{k}}^{\perp}V$ is orthogonal to $JTM^{2}$.
We have
\begin{eqnarray*}
\langle JE_{j}, \nabla _{E_{k}}^{\perp}V\rangle
&=&
-\langle \nabla _{E_{k}}^{\perp}JE_{j}, V\rangle
\\
&=&
-\langle \overline{\nabla}_{E_{k}} JE_{j},V\rangle -\langle A_{JE_{j}}E_{k},V\rangle
\\
&=&
-\langle J \overline{\nabla}_{E_{k}}E_{j},V\rangle
\\
&=&
-\langle J \nabla_{E_{k}}E_{j},V\rangle -\langle JB(E_{k}, E_{j}),V\rangle 
\\
&=&
0.
\end{eqnarray*} 

In order to prove that 
$\nabla_{E_{k}}^{\perp}V$ is orthogonal to $\im B$ we set
$$
A_{ijk}=-\langle B(E_{i},E_{j}),\nabla_{E_{k}}^{\perp}V\rangle = \langle \nabla_{E_{k}}^{\perp}B(E_{i},E_{j}),V\rangle
$$
and we will prove that $A_{ijk}=0$. Since $B$ is symmetric, we obtain $A_{ijk}=A_{jik}$.

Now, we notice that
\begin{eqnarray*}
\langle (\nabla_{E_{k}}^{\perp}B)(E_{i},E_{j}),V\rangle
&=&
\langle \nabla _{E_{k}}^{\perp} B(E_{i},E_{j}),V\rangle -\langle B(\nabla_{E_{k}}E_{i},E_{j}),V\rangle 
\\
&\ &
-\langle B(E_{i},\nabla_{E_{k}}E_{j}),V\rangle
\\
&=&
\langle \nabla_{E_{k}}^{\perp}B(E_{i},E_{j}),V\rangle
\\
&=&
A_{ijk}.
\end{eqnarray*}

Using the Codazzi Equation (\ref{the Codazzi Equation}) we have 
\begin{eqnarray*}
A_{ijk}
&=&
\langle (\nabla_{E_{k}}^{\perp}B)(E_{i},E_{j}),V\rangle
\\
&=&
\langle (\nabla_{E_{i}}^{\perp}B)(E_{k},E_{j})+(\overline{R}(E_{k},E_{i})E_{j})^{\perp}
,V \rangle
\\
&=&
\langle (\nabla_{E_{i}}^{\perp}B)(E_{k},E_{j}),V\rangle
=
A_{kji}
\\
&=&
\langle (\nabla_{E_{i}}^{\perp}B)(E_{j},E_{k}),V\rangle
\\
&=&
\langle (\nabla_{E_{j}}^{\perp}B)(E_{i},E_{k})+(\overline{R}(E_{j},E_{i})E_{k})^{\perp}
,V \rangle
\\
&=&
\langle (\nabla_{E_{j}}^{\perp}B)(E_{i},E_{k}),V\rangle
=
A_{ikj},
\end{eqnarray*}
that shows that $A_{ijk}=A_{kji}=A_{ikj}$.

Next, since the normal vector field
$\nabla_{E_{k}}^{\perp}V$ is orthogonal to $JTM^{2}$, from Corollary \ref{Normalizing} it follows that the  basis
$\{E_{1},E_{2}\}$ diagonalizes $A_{\nabla_{E_{k}}^{\perp}V}$ as well, and we get 
\begin{eqnarray*}
A_{ijk}&=&-\langle B(E_{i},E_{j}),\nabla_{E_{k}}^{\perp}V\rangle 
=
-\langle (A_{\nabla_{E_{k}}^{\perp}V})E_{i},E_{j}\rangle
=
-\langle \lambda_{i}E_{i}, E_{j}\rangle
\\
&=&
0,
\end{eqnarray*}
for $i \neq j$. Hence, $A_{ijk}=0$ if two indices are different from each other.

Finally, we have 
\begin{eqnarray*}
A_{iii}&=&-\langle B(E_{i},E_{i}),\nabla_{E_{i}}^{\perp}V \rangle
\\
&=&
-\langle 2H,\nabla_{E_{i}}^{\perp}V\rangle +\langle B(E_{j},E_{j}), \nabla_{E_{i}}^{\perp}V\rangle \ \quad (j\neq i)
\\
&=&
2\langle  \nabla_{E_{i}}^{\perp}H, V\rangle -A_{jji}
\\
&=&
0.
\end{eqnarray*}
Thus $A_{ijk}=0$.

Now, if $V$ is normal to $M^{2}$ and orthogonal to $L$, it follows that $JV$ is normal and orthogonal to $L$. Further, we have 
\begin{eqnarray*}
\langle (JB(E_{i},E_{j}))^{\perp}, \nabla_{E_{k}}^{\perp}V\rangle &=& -\langle \overline{\nabla}_{E_{k}}(JB(E_{i},E_{j}))^{\perp}, V \rangle
\\
&=&
 -\langle \overline{\nabla}_{E_{k}}JB(E_{i},E_{j}), V \rangle + \langle \overline{\nabla}_{E_{k}}(JB(E_{i},E_{j}))^{\top}, V \rangle
\\
&=&
\langle JA_{B(E_{i},E_{j})}E_{k},V\rangle - \langle J \nabla_{E_{k}}^{\perp}B(E_{i},E_{j}),V\rangle 
\\
&\ &
+\langle B(E_{k},(JB(E_{i},E_{j}))^{\top}),V\rangle
\\
&=&
\langle \nabla_{E_{k}}^{\perp}B(E_{i},E_{j}),JV \rangle
\\
&=&
0,
\end{eqnarray*}
and we conclude.

Finally, we still need to prove that $L$ has real dimension less or equal to $6$.
\\
Indeed, since $\{ JE_{1},JE_{2}\}$ is a local orthogonal-unit system in $NM^{2}$, we  can consider
$$
\{ JE_{1},JE_{2},V_{1},\ldots, V_{2n-4}\},
$$
a local orthonormal frame field in $NM^{2}$.
We have,
$$
B(E_{1},E_{2})=\alpha JE_{1}+\beta JE_{2}+\gamma_{1}V_{1}+\cdots+\gamma_{2n-4}V_{2n-4},
$$
and then 
\begin{eqnarray*}
\langle B(E_{1},E_{2}), V_{1}\rangle &=& \gamma_{1}
=\langle A_{V_{1}}E_{1},E_{2}\rangle
\\
&=&
0.
\end{eqnarray*}
Therefore, $\gamma_{i}=0$, for any  $i=1, \dots,2n-4$, so 
$B(E_{1},E_{2})=\alpha JE_{1}+\beta JE_{2}$ and $JB(E_{1},E_{2})=-\alpha E_{1}-\beta E_{2}$.
Let $X,Y \in C(TM^{2})$. We have
\begin{eqnarray*}
B(X,Y)&=& B(X^{1}E_{1}+X^{2}E_{2},Y^{1}E_{1}+Y^{2}E_{2})
\\
&=&
X^{1}Y^{1}B(E_{1},E_{1})+(X^{1}Y^{2}+Y^{1}X^{2})B(E_{1},E_{2})+X^{2}Y^{2}B(E_{2},E_{2}).
\end{eqnarray*}
\begin{enumerate}
\item 
As $J$ and $\perp$ are linear, we have 
\begin{eqnarray*}
(JB(X,Y))^{\perp}&=&X^{1}Y^{1}(JB(E_{1},E_{1}))^{\perp}+(X^{1}Y^{2}+X^{2}Y^{1})(JB(E_{1},E_{2}))^{\perp}
\\
&\ &
+X^{2}Y^{2}(JB(E_{2},E_{2}))^{\perp}
\\
&=&
X^{1}Y^{1}(JB(E_{1},E_{1}))^{\perp}+(X^{1}Y^{2}+X^{2}Y^{1})(-\alpha E_{1}-\beta E_{2})^{\perp}
\\
&\ &
+X^{2}Y^{2}(JB(E_{2},E_{2}))^{\perp}
\\
&=&
X^{1}Y^{1}(JB(E_{1},E_{1}))^{\perp}+X^{2}Y^{2}(JB(E_{2},E_{2}))^{\perp}.
\end{eqnarray*}
Thus, $(J\im B)^{\perp}$ is utmost of real dimension equal to $2$.
\item
Next, we consider the normal vector $B(X,Y)+JZ$, and we have
\begin{eqnarray*}
B(X,Y)+JZ &=& X^{1}Y^{1}B(E_{1},E_{1})+(X^{1}Y^{2}+Y^{1}X^{2})B(E_{1},E_{2})+X^{2}Y^{2}B(E_{2},E_{2})
\\
&\ &
+J(Z_{1}E_{1}+Z_{2}E_{2})
\\
&=&
X^{1}Y^{1}B(E_{1},E_{1})+(X^{1}Y^{2}+Y^{1}X^{2})(\alpha JE_{1}+\beta JE_{2})
\\
&\ &
+X^{2}Y^{2}B(E_{2},E_{2})
+Z_{1}JE_{1}+Z_{2}JE_{2}
\\
&=&
X^{1}Y^{1}B(E_{1},E_{1})+(\alpha X^{1}Y^{2}+\alpha Y^{1}X^{2}+Z_{1}) JE_{1}
\\
&\ &
+(\beta X^{1}Y^{2}+\beta Y^{1}X^{2}+Z_{2}) JE_{2} +X^{2}Y^{2}B(E_{2},E_{2}).
\end{eqnarray*}
Then $\Span \{ \im B \cup J(TM^{2}) \} $ is utmost of real dimension equal to $4$.
\end{enumerate}
Therefore, $L$ has real dimension less or equal to $6$.
\end{proof}

\begin{lemma}\label{lemma L}
Assume that $H$ is nowhere an umbilical direction. Denote by $\tilde{L}=L\oplus TM^{2}$, then 
$\tilde{L}$ is parallel with respect to the Levi-Civita connection on the complex space form $N^{n}(c)$ and it is invariant by the curvature tensor $\overline{R}$, i.e., $\overline{R}(\overline{u},\overline{v})\overline{w} \in \tilde{L}$, for all $\overline{u},\overline{v},\overline{w}\in \tilde{L}$.
\end{lemma}

\begin{proof}
From Proposition \ref{Reduction of codimension for biconservative surfaces}, it is easy to see that $\tilde{L}$ is parallel with respect to the Levi-Civita connection 
$\overline{\nabla}$ on the complex space form $N^{n}(c)$. Indeed, if 
$\sigma \in C(L) \subset C(\tilde{L})$, we have 
$$
\overline{\nabla}_{X}\sigma =\nabla_{X}^{\perp}\sigma-A_{\sigma}X
$$
as $\nabla_{X}^{\perp}\sigma \in C(L)$ and $A_{\sigma}X \in C(TM^{2})$, we obtain $\overline{\nabla}_{X}\sigma \in C(\tilde{L})$.
\\
Also, if $Y \in C(TM^{2})\subset C(\tilde{L})$, we have 
$$
\overline{\nabla}_{X}Y=\nabla_{X}Y+B(X,Y),
$$
and since $\nabla_{X}Y\in C(TM^{2})$ and $B(X,Y)\in C(L)$, we get $\overline{\nabla}_{X}Y \in C(\tilde{L})$.
\\
Now, in order to show that 
$\tilde{L}$ is invariant by the curvature tensor 
$\overline{R}$, we need first to prove that $J\tilde{L}\subset \tilde{L}$, which implies $J\tilde{L}= \tilde{L}$. 
\begin{enumerate}
\item \label{1'} Let $X\in TM^{2}\subset \tilde{L}$. By the definition of $L$, we obtain $JX\in L\subset \tilde{L}$.
\item Let $B(X,Y)\in L\subset\tilde{L}$. We have
$$
JB(X,Y)=(JB(X,Y))^{\top}+(JB(X,Y))^{\perp},
$$
and since $(JB(X,Y))^{\top}\in TM^{2}\subset \tilde{L}$ and  
$(JB(X,Y))^{\perp}\in L\subset \tilde{L}$, we get $JB(X,Y)\in \tilde{L}$, for all $X,Y\in TM^{2}$.
\item Let $(JB(X,Y))^{\perp} \in L\subset \tilde{L}$, 
$$
J((JB(X,Y))^{\perp})
=J(JB(X,Y)-(JB(X,Y))^{\top}).
$$
Take $Z=(JB(X,Y))^{\top}$, then 
$$
J((JB(X,Y))^{\perp})
=-B(X,Y)-JZ,
$$
and since $B(X,Y)\in L \subset \tilde{L}$ and $JZ\in L \subset\tilde{L}$, we get $J((JB(X,Y))^{\perp})\in \tilde{L}$, for all $X,Y\in TM^{2}$.
\item Let $JX\in L \subset\tilde{L}$, 
$$
J(JX)=-X.
$$
thus $J(JX)\in \tilde{L}$, for all $X\in TM^{2}$.
\end{enumerate} 
Therefore, $J\tilde{L}\subset\tilde{L}$ and so $J\tilde{L}= \tilde{L}$.
Now, we have 
\begin{eqnarray*}
\overline{R}(\overline{u},\overline{v})\overline{w}=\frac{c}{4}\Big \{ \langle\overline{v}, \overline{w}\rangle \overline{u}- \langle\overline{u}, \overline{w}\rangle \overline{v}+ \langle J\overline{v},  \overline{w}\rangle J \overline{u}-\langle J\overline{u}, \overline{w}\rangle J \overline{v} +2 \langle J \overline{v}, \overline{u}\rangle J\overline{w} \Big\},
\end{eqnarray*}
hence $\overline{R}(\overline{u},\overline{v})\overline{w}\in\tilde{L}$ for all $\overline{u},\overline{v},\overline{w}\in \tilde{L}$.
\end{proof} 
Now we can state the main result of this section:

\begin{theorem}\label{Reduction theorem}
Let $M^{2}$ be a non pseudo-umbilical PMC totally real surface in a complex space form $N^{n}(c)$, $c\neq 0$, $n\geq 4$. Then there exists a totally geodesic complex submanifold $N^{4}(c)\subset N^{n}(c)$ such that $M^{2}\subset N^{4}(c)$.  
\end{theorem}

\begin{proof}
In the first case, we assume that $M^{2}$ is nowhere pseudo-umbilical. We can apply \cite [Theorem 2]{Eschenburg-Tribuzy} and \cite[Theorem 1]{Ferreira-Tribuzy} to the subbundle $\tilde{L}$ as defined in Lemma \ref{lemma L} and  Proposition \ref{Reduction of codimension for biconservative surfaces} to conclude that there exists a totally geodesic submanifold $N'$ of $N^{n}(c)$ such that $M\subset N'$ and 
$\tilde{L}_{p}=T_{p}N'$ for all $p\in M^{2}$. Since $J\tilde{L}= \tilde{L}$, $N'$ is a complex space form  $N'=N^{4}(c)$ (see \cite{Chen-Nagano1, Chen-Nagano2}).

In the second case, assume that $M^{2}$ admits some pseudo-umbilical points and denote by $W$ the set of all non pseudo-umbilical points of $M^{2}$. As $M^{2}\backslash W$ has no accumulation points, the subset $W$ is open, dense and connected. 

In order to prove that $W$ is connected, we show that $W$ is path-connected. Indeed, let $p, q\in W$, thus, $p,q\in M^{2}$ and there exists a path $\gamma$ in $M^{2}$ that joins 
$p$ and $q$. If the path is already in $W$, we conclude directly. Assume that $\gamma$ passes through at least one (but finite number of) pseudo-umbilical point. Denote such a point by $z$. Since pseudo-umbilical points are isolated, we choose a neighborhood of $z$ containing no other pseudo-umbilical point. In this neighborhood we can smoothly modify the curve $\gamma$ to avoid the point $z$. In this way, we obtain a new path that joins $p$ and $q$ and lies in $W$.

We apply the same argument as in the first case to $W$ to conclude that $W\subset N^{4}(c)$. Then, by a standard argument, we conclude that the whole $M^{2}$ lies in that $N^{4}(c)$. More precisely, let $z\in M^{2}\backslash W$. Since $z$ is isolated, there exists a sequence $\{ p_{n} \}_{n\in \mathbb {N}^{*}}$ that converges to $z$ in $M^{2}$, $p_{n}\neq z$,  and $p_{n}$ belongs to the neighborhood of $z$  that isolates $z$ from the other points of $M^{2}\backslash W$. As $\{p_{n}\}_{n\in \mathbb{N}^{*}}$ converges to $z$ in $M^{2}$, it follows that 
$\{ p_{n} \}_{n\in \mathbb{N}^{*}}$ is a Cauchy sequence in $M^{2}$, or in $W$, and from here it follows that $\{ p_{n} \}_{n\in \mathbb {N}^{*}}$ is a Cauchy sequence also in $N^{4}(c)$. But $N^{4}(c)$ is complete, so $\{p_{n}\}_{n\in \mathbb{N}^{*}}$ converges to some point $z'\in N^{4}(c)$ in $N^{4}(c)$. However, as $\{p_{n}\}_{n\in \mathbb{N}^{*}}$ converges to $z$ in $N^{n}(c)$ and $\{p_{n}\}_{n\in \mathbb{N}^{*}}$ converges to $z'$ in $N^{n}(c)$, we get $z=z'$, thus $z\in N^{4}(c)$.
\end{proof}

\begin{remark}
When $M^{2}$ is pseudo-umbilical and a topological sphere, then the situation is different. First, we recall that if $M^{2}$ is a topological sphere, CMC and biconservative in an arbitrary Riemannian manifold, then it is pseudo-umbilical (see \cite[Corollary 4.3]{M.O.R.}). Now, according to the result in \cite{BO}, when $M^{2}$ is a PMC totally real surface in a complex space form $N^{n}(c)$, $c\neq 0$, and $M^{2}$ is a topological sphere, then there exists a  totally geodesic totally real submanifold $N^{'}$ such that $M^{2}\subset N^{'}$. We note that the technique used in \cite{BO} is completely different from that used here.
\end{remark}

We can improve the above result and reduce the codimension even more, under a slightly stronger assumption. 

\begin{theorem}
Let $M^{2}$ be a non pseudo-umbilical PMC totally real surface in a complex space form $N^{n}(c)$, $c\neq 0$. If $H \in C(JTM^{2})$, then there exists a totally geodesic complex submanifold $N^{2}(c)\subset N^{n}(c)$ such that $M^{2}\subset N^{2}(c)$.  
\end{theorem}

\begin{proof}
We will assume that $M^{2}$ is nowhere pseudo-umbilical, otherwise we follow the argument of the second part of the proof of  Theorem \ref{Reduction theorem}. Consider $\{E_{1},E_{2}\}$ a local orthonormal frame field that diagonalizes simultaneously $A_{H}$ and $A_{V}$.

In the proof of Proposition \ref{Reduction of codimension for biconservative surfaces} we have seen that $B(E_{1},E_{2})\in C(JTM^{2})$. We will prove that 
$$
B(E_{1},E_{1}), B(E_{2},E_{2})\in C(J(TM^{2})).
$$
Thus, 
$$
\im B \subset J(TM^{2}), \quad J\im B \subset TM^{2}
$$
and so $L=J(TM^{2})$.

Indeed, let $V$ orthogonal to $J(TM^{2})$ and normal to $M^{2}$. Then $V\perp H$. We have 
$\langle V,JE_{i}\rangle =0$, so 
$\langle JV,E_{i}\rangle =0$ and therefore $JV$ is normal to $M^{2}$. And since $\langle JV,JE_{i}\rangle =\langle V,E_{i}\rangle=0$, we also get $JV\perp JTM^{2}$.

We have 
\begin{eqnarray}\label{1st eq}
\overline{\nabla}_{E_{i}}JE_{j}=-A_{JE_{j}}E_{i}+\nabla^{\perp}_{E_{i}}JE_{j},
\end{eqnarray}
on the other hand,
\begin{eqnarray}\label{2nd eq}
\overline{\nabla}_{E_{i}}JE_{j}=J\overline{\nabla}_{E_{i}}E_{j}= J\nabla_{E_{i}}E_{j}+JB(E_{i},E_{j}).
\end{eqnarray}
Taking the inner product of Equations (\ref{1st eq}) and (\ref{2nd eq}) with $V$, we obtain
\begin{eqnarray}\label{inner product withV}
\langle JB(E_{i},E_{j}),V\rangle =\langle \nabla^{\perp}_{E_{i}}JE_{j}, V  \rangle , \quad \forall i,j=1,2.
\end{eqnarray}

Now, as $H\in J(TM^{2})$, we obtain 
$$
H=\langle H,JE_{1}\rangle JE_{1}+ \langle H,JE_{2} \rangle JE_{2}.
$$
And, since $M^{2}$ is PMC, we get
\begin{eqnarray}\label{nabla-perp-H}
0=\nabla^{\perp}_{X}H &=&
X(\langle H,JE_{1}\rangle)JE_{1}+\langle H,JE_{1}\rangle \nabla_{X}^{\perp}JE_{1}
\nonumber
\\
&\ &
+X(\langle H,JE_{2}\rangle)JE_{2}+\langle H,JE_{2}\rangle \nabla_{X}^{\perp}JE_{2}.
\end{eqnarray}
Taking the inner product of Equation (\ref{nabla-perp-H}) with $V$, we obtain 
\begin{eqnarray}
\langle H,JE_{1}\rangle \langle \nabla_{X}^{\perp}JE_{1},V\rangle +\langle H,JE_{2}\rangle \langle \nabla_{X}^{\perp} JE_{2},V\rangle =0.
\end{eqnarray}
Using Equation (\ref{inner product withV}) and taking $X=E_{1}$, as  $JB(E_{1},E_{2})$ is tangent, we get
\begin{eqnarray*}
\langle H,JE_{1}\rangle \langle JB(E_{1},E_{1}),V\rangle &=&-\langle H,JE_{2}\rangle \langle \nabla_{E_{1}}^{\perp} JE_{2},V\rangle 
\\
&=&
\langle H,JE_{2}\rangle \langle JB(E_{1},E_{2}),V\rangle
\\
&=&
0.
\end{eqnarray*}
Hence, $\langle H,JE_{1}\rangle=0$ or $\langle JB(E_{1},E_{1}),V\rangle=0$.

Let $p\in M^{2}$. If at the point $p$ we have $\langle JB(E_{1},E_{1}),V\rangle =0$, then $\langle B(E_{1},E_{1}),JV\rangle=0$, so $\langle B(E_{1},E_{1}),U\rangle =0$ for all $U$ normal to $M^{2}$ and $U\perp J(T_{p}M^{2})$. Therefore, $B(E_{1},E_{1})\in J(T_{p}M^{2})$ and, as $H\in J(TM^{2})$, we also get $B(E_{2},E_{2})\in J(T_{p}M^{2})$.

Now, assume that at $p$, $\langle JB(E_{1},E_{1}),V\rangle \neq 0$. Thus, $\langle H,JE_{1}\rangle = 0$ around $p$. Then $E_{2}=JH/\vert H\vert$ and so $JE_{2}=-H/\vert H \vert$. Hence,
\begin{eqnarray}\label{eq1'}
\overline{\nabla}_{E_{2}}JE_{2}=-A_{JE_{2}}E_{2}+\nabla^{\perp}_{E_{2}}(\tfrac{-H}{\vert H \vert}).
\end{eqnarray}
On the other hand,
\begin{eqnarray}\label{eq2'}
\overline{\nabla}_{E_{2}}JE_{2}=J\overline{\nabla}_{E_{2}}E_{2}= J\nabla_{E_{2}}E_{2}+JB(E_{2},E_{2}).
\end{eqnarray}
Taking the inner product of Equations (\ref{eq1'}) and (\ref{eq2'}) with $V$, we obtain 
$$
\langle JB(E_{2},E_{2}), V \rangle =0,
$$
for any $V$ normal to $M^{2}$ and orthogonal to $JTM^{2}$.
Thus, $\langle B(E_{2},E_{2}), JV \rangle =0$, then $\langle B(E_{2},E_{2}),U\rangle =0$ for all $U$ normal to $M^{2}$ and orthogonal to $JTM^{2}$. Therefore, $B(E_{2},E_{2})\in J(TM^{2})$ and, as the mean curvature vector field $H\in J(TM^{2})$, we get that $B(E_{1},E_{1})\in J(TM^{2})$.
\end{proof}

\section{Further studies}

In this section, we introduce examples where we use the Segre embedding (see for example \cite{Chen-Segre, Segre}) to show the existence of biconservative CMC submanifolds of the complex projective space which are non-PMC. This result shows that when the dimension of a biconservative submanifold is greater than two, the situation is less rigid and one can expect to find many interesting examples other than the totally real ones. Moreover, from the above examples we determine which of them are proper-biharmonic.

\begin{theorem}
Let $\gamma$ be a curve of non-zero curvature $\kappa$ in the complex projective space 
$\mathbb{C}P^1(4)$ of complex dimension $1$. Then, we have
\begin{enumerate}
\item via the Segre embedding of $\mathbb{C}P^{1}(4)\times \mathbb{C}P^{q}(4)$ into 
$\mathbb{C}P^{1+2q}(4)$, the product $M^{1+2q}=\gamma \times \mathbb{C}P^q (4)$ is a  biconservative submanifold of 
$\mathbb{C}P^{1+2q}(4)$ 
if and only if $\kappa=constant$; in this case, $M^{1+2q}$ is CMC non-PMC, and moreover, it is not totally real;
\item $M^{1+2q}$ is a  proper-biharmonic submanifold of 
$\mathbb{C}P^{1+2q}(4)$ if and only if 
$\kappa^{2}=4$, i.e., $\gamma$ is proper-biharmonic in $\mathbb{C}P^{1}(4)$.
\end{enumerate}  
\end{theorem}

\begin{proof}
Let 
$$
\gamma :I\to\mathbb{C}P^{1}(4),
$$
where $\gamma$ is parametrized by arc-length, with non-zero constant  curvature $\kappa$, and we identify $\mathbb{C}P^{1}(4)$ with the sphere
$\mathbb{S}^{2}$ of curvature $4$.

Further, consider the following two embeddings
$$
\textbf{i}:M^{1+2q}=\gamma \times \mathbb{C}P^{q}(4)\to \mathbb{C}P^{1}(4)\times \mathbb{C}P^{q}(4)
$$
and 
$$
\textbf{j}:\mathbb{C}P^{1}(4)\times \mathbb{C}P^{q}(4)\to \mathbb{C}P^{1+2q}(4),
$$
where $\textbf{j}$ is the Segre embedding, and let $\phi=\textbf{j}\circ \textbf{i}$ be the composition map. We have 
$$
B^{\phi}(X,Y)= B^{\textbf{i}}(X,Y)+B^{\textbf{j}}(X,Y),
$$
for all $X,Y\in C(TM^{1+2q})$, and thus 
\begin{eqnarray*}
H^{\phi}=H^{\textbf{i}}+\frac{1}{1+2q}\sum_{l=1}^{1+2q}B^{\textbf{j}}(E_{l},E_{l}),
\end{eqnarray*}
where $E_{1}=\gamma^{'}$ and $\{E_{2},\ldots,E_{1+2q}\}$ is a local orthonormal frame field defined on $\mathbb{C}P^{q}(4)$ and tangent to $\mathbb{C}P^{q}(4)$.

Recall that \cite{Chen-Segre}, since $\textbf{j}$ is the Segre embedding, we have
$$
B^{\textbf{j}}(E_{l},E_{l})=0, \quad \forall l=1,\ldots ,1+2q.
$$
Therefore,
\begin{eqnarray*}
H^{\phi}=H^{\textbf{i}}.
\end{eqnarray*}

Now, we have 
\begin{eqnarray}\label{nabla H with A}
\nabla_{X}^{\mathbb{C}P^{1+2q}}H^{\phi}=\nabla_{X}^{\perp_{\phi}}H^{\phi}-A^{\phi}_{H^{\phi}}X,
\end{eqnarray}
and, on the other hand,
\begin{eqnarray}\label{nabla H}
\nabla_{X}^{\mathbb{C}P^{1+2q}}H^{\phi}&=& \nabla_{X}^{\mathbb{C}P^{1+2q}}H^{\textbf{i}}\nonumber
\\
&=&
\nabla_{X}^{\mathbb{C}P^{1}\times\mathbb{C}P^{q}}H^{\textbf{i}}+B^{\textbf{j}}(X,H^{\textbf{i}}).
\end{eqnarray}

In order to compute $H^{\textbf{i}}$, we consider $\{E_{1}=\gamma^{'}={\bf t},{\bf n}\}$ the Frenet frame field along 
$\gamma$ in $\mathbb{C}P^{1}(4)$. Since
\begin{eqnarray*}
B^{\textbf{i}}(E_{1},E_{1})
&=&
\nabla_{E_{1}}^{\mathbb{C}P^{1}\times \mathbb{C}P^{q}}E_{1}-\nabla^{M^{1+2q}}_{E_{1}}E_{1}
=
\nabla_{E_{1}}^{\mathbb{C}P^{1}}E_{1}
\\
&=&
\kappa \bf{n}
\end{eqnarray*} 
and 
\begin{eqnarray*}
B^{\textbf{i}}(E_{l},E_{l})
&=&
\nabla_{E_{l}}^{\mathbb{C}P^{1}\times \mathbb{C}P^{q}}E_{l}-\nabla^{M^{1+2q}}_{E_{l}}E_{l}
\\
&=&
0
\end{eqnarray*}
for $l=2,\ldots,1+2q$, we obtain
\begin{eqnarray}\label{Hi}
H^{\textbf{i}}=\frac{1}{1+2q}\kappa \bf{n}.
\end{eqnarray} 
Replacing (\ref{Hi}) in (\ref{nabla H}), we get
\begin{eqnarray}\label{nabla H-phi}
\nabla_{X}^{\mathbb{C}P^{1+2q}}H^{\phi}=\frac{\kappa}{1+2q}\{\nabla_{X}^{\mathbb{C}P^{1}\times\mathbb{C}P^{q}}{\bf n}+B^{\textbf{j}}(X,{\bf n})\}+\frac{X(\kappa)}{1+2q}\textbf{n}.
\end{eqnarray}
To find the shape operator $A^{\phi}_{H^{\phi}}$ and to prove that the immersion $\phi$ is not PMC, we consider first $X=E_{1}$ in Equation (\ref{nabla H-phi}), and we obtain
\begin{eqnarray}\label{nabla H-with X}
\nabla_{E_{1}}^{\mathbb{C}P^{1+2q}}H^{\phi}
&=&
\frac{\kappa}{1+2q}\{\nabla_{E_{1}}^{\mathbb{C}P^{1}}{\bf n}+B^{\textbf{j}}(E_{1},{\bf n})\}+\frac{\kappa ^{'}}{1+2q}\textbf{n}\nonumber
\\
&=&
\frac{\kappa}{1+2q}\{-\kappa {\bf t}\}+ \frac{\kappa ^{'}}{1+2q}\textbf{n}
\nonumber
\\
&=&
-\frac{\kappa^{2}}{1+2q}{\bf t}+\frac{\kappa ^{'}}{1+2q}\textbf{n}
\nonumber
\\
&=&
-\frac{\kappa^{2}}{1+2q}E_{1}+\frac{\kappa ^{'}}{1+2q}\textbf{n}.
\end{eqnarray}
From Equations (\ref{nabla H with A}) and (\ref{nabla H-with X}), we conclude
\begin{eqnarray}\label{A and nabla H1}
\nabla_{E_{1}}^{\perp_{\phi}}H^{\phi}=\frac{\kappa ^{'}}{1+2q}\textbf{n} \quad \textnormal{and} \quad A^{\phi}_{H^{\phi}}E_{1}=\frac{\kappa^{2}}{1+2q}E_{1}.
\end{eqnarray}
Second, if $X=E_{l}$, $l=2,\ldots,2q+1$, from Equation (\ref{nabla H-phi}), we get 
\begin{eqnarray}\label{nabla H-with B}
\nabla_{E_{l}}^{\mathbb{C}P^{1+2q}}H^{\phi}
&=&
\frac{\kappa}{1+2q}\{\nabla_{E_{l}}^{\mathbb{C}P^{1}\times\mathbb{C}P^{q}}{\bf n}+B^{\textbf{j}}(E_{l},{\bf n})\}\nonumber
\\
&=&
\frac{\kappa}{1+2q}B^{\textbf{j}}(E_{l},{\bf n}).
\end{eqnarray}
Using Equations (\ref{nabla H with A}) and (\ref{nabla H-with B}), we conclude
\begin{eqnarray}\label{nablaXl}
\nabla_{E_{l}}^{\perp_{\phi}}H^{\phi}=\frac{\kappa}{1+2q}B^{\textbf{j}}(E_{l},{\bf n}) \quad \textnormal{and} \quad A^{\phi}_{H^{\phi}}E_{l}=0.
\end{eqnarray}
Since $\vert B^{\textbf{j}}(E_{1},E_{l})\vert=1$, for all $l=2,\ldots, 1+2q$, $B^{\textbf{j}}(E_{l},{\bf n})\neq 0$.
Therefore, 
$\nabla_{E_{l}}^{\perp_{\phi}}H^{\phi}\neq 0$, i.e. $M^{1+2q}$ is a non-PMC submanifold of
$\mathbb{C}P^{1+2q}(4)$. 

In the following, in order to study the biconservativity of 
$\phi$, we compute the curvature term in the biconservative equation. We have  
\begin{eqnarray*}
R^{\mathbb{C}P^{1+2q}}(X,H^{\phi})X
&=&
\langle H^{\phi},X\rangle X-\langle X,X\rangle H^{\phi}+\langle JH^{\phi},X\rangle JX
\\
&\ & -\langle JX,X\rangle JH^{\phi}+2\langle JH^{\phi},X\rangle JX 
\\
&=&-H^{\phi}+3\langle JH^{\phi},X\rangle JX,
\end{eqnarray*}
for all $X\in C(TM^{1+2q})$.
Then 
\begin{eqnarray*}
\trace (R^{\mathbb{C}P^{1+2q}}(\cdot,H^{\phi})\cdot)^{\top ^{\phi}}
&=&
3\{\trace \langle T,\cdot\rangle J\cdot\}^{\top^{\phi}}
\\
&=&
3(JT)^{\top^{\phi}},
\end{eqnarray*}
where $JH^{\phi}=T+N$ with respect to $\phi$.

In our case,
$$
JH^{\phi}=JH^{\textbf{i}}=\frac{\kappa}{1+2q}(\pm\textbf{t})\in C(TM^{1+2q}),
$$
so $T=JH^{\phi}$ and then $JT=-H^{\phi}$ which implies
$(JT)^{\top^{\phi}}=0$. Therefore,
$$
\trace (R^{\mathbb{C}P^{1+2q}}(\cdot,H^{\phi})\cdot)^{\top ^{\phi}}=0.
$$

Now, according to Proposition \ref{Biconservativity Cdts}, in order to show that $\phi$ is biconservative, we must prove 
\begin{eqnarray}\label{Biconservative-CPn}
4\trace A^{\phi}_{\nabla_{(\cdot)}^{\perp_{\phi}}H^{\phi}}(\cdot)+(1+2q)\grad (\vert H^{\phi}\vert^{2})=0.
\end{eqnarray}
The second term of the left hand side of \eqref{Biconservative-CPn} can be written as 
\begin{eqnarray}\label{Grad H}
(1+2q)\grad (\vert H^{\phi}\vert^{2})&=& \frac{2\kappa \kappa ^{'}}{1+2q}E_{1}.
\end{eqnarray}
For the first term, by (\ref{A and nabla H1}) we have
\begin{eqnarray}\label{A the shape operator of M}
A^{\phi}_{\nabla_{E_{1}}^{\perp_{\phi}}H^{\phi}}E_{1}&=& A^{\phi}_{\frac{\kappa^{'}}{1+2q}\textbf{n}}E_{1}
\nonumber
=
\frac{\kappa ^{'}}{\kappa}A^{\phi}_{H^{\phi}}E_{1}
\nonumber
\\
&=&
\frac{\kappa \kappa ^{'}}{1+2q}E_{1}.
\end{eqnarray}
Next, if $l=2,\ldots,1+2q$ and $k=1,\ldots,1+2q$, we have
\begin{eqnarray*}
\langle A^{\phi}_{\nabla_{E_{l}}^{\perp_{\phi}}H^{\phi}}E_{l},E_{k}\rangle 
&=&
\frac{\kappa}{1+2q}\langle A^{\phi}_{B^{\textbf{j}}(E_{l},\textbf{n})}E_{l},E_{k}\rangle
\\
&=&
\frac {\kappa}{1+2q}\langle B^{\phi}(E_{l},E_{k}),
B^{\textbf{j}}(E_{l},\textbf{n})
\rangle
\\
&=&
\frac {\kappa}{1+2q}\langle B^{\textbf{i}}(E_{l},E_{k})+B^{\textbf{j}}(E_{l},E_{k}),
B^{\textbf{j}}(E_{l},\textbf{n})
\rangle
\\
&=&
\frac {\kappa}{1+2q}\{\langle B^{\textbf{i}}(E_{l},E_{k}),
B^{\textbf{j}}(E_{l},\textbf{n})
\rangle +\langle B^{\textbf{j}}(E_{l},E_{k}),
B^{\textbf{j}}(E_{l},\textbf{n})
\rangle \}
\\
&=&
\frac {\kappa}{1+2q}\langle B^{\textbf{j}}(E_{l},E_{k}),
B^{\textbf{j}}(E_{l},\textbf{n})
\rangle.
\end{eqnarray*}
Further, from the Gauss Equation \eqref{the Gauss Equation} for the immersion \textbf{j} we have  
\begin{eqnarray*}
\langle B^{\textbf{j}}(E_{l},E_{k}),
B^{\textbf{j}}(E_{l},\textbf{n})
\rangle &=& \langle B^{\textbf{j}}(\textbf{n},E_{k}), B^{\textbf{j}}(E_{l},E_{l})
\rangle - \langle R^{\mathbb{C}P^{1+2q}}(E_{l},\textbf{n})E_{l},E_{k}\rangle 
\\
&\ &
+ \langle R^{\mathbb{C}P^{1}\times \mathbb{C}P^{q}}(E_{l},\textbf{n})E_{l},E_{k}\rangle.
\end{eqnarray*}
We have
$$
B^{\textbf{j}}(E_{l},E_{l})=0.
$$
Since
\begin{eqnarray*}
R^{\mathbb{C}P^{1+2q}}(E_{l},\textbf{n})E_{l}
&=&
\langle \textbf{n},E_{l}\rangle E_{l}-\langle E_{l},E_{l}\rangle \textbf{n}+\langle J\textbf{n},E_{l}\rangle JE_{l}
\\
&\ & -\langle JE_{l},E_{l}\rangle J\textbf{n}+2\langle J\textbf{n},E_{l}\rangle JE_{l} 
\\
&=&
-\textbf{n},
\end{eqnarray*}
it follows
$$
\langle R^{\mathbb{C}P^{1+2q}}(E_{l},\textbf{n})E_{l},E_{k}\rangle =0.
$$
Also, we have
\begin{eqnarray*}
 R^{\mathbb{C}P^{1}\times \mathbb{C}P^{q}}(E_{l},\textbf{n})E_{l} &=& \nabla_{E_{l}}^{\mathbb{C}P^{1}\times \mathbb{C}P^{q}}\nabla_{\textbf{n}}^{\mathbb{C}P^{1}\times \mathbb{C}P^{q}} E_{l}-\nabla_{\textbf{n}}^{\mathbb{C}P^{1}\times \mathbb{C}P^{q}}\nabla_{E_{l}} ^{\mathbb{C}P^{1}\times \mathbb{C}P^{q}}E_{l}
\\
&\ &
-\nabla_{[E_{l},\textbf{n}]}^{\mathbb{C}P^{1}\times \mathbb{C}P^{q}} E_{l}
\\
&=&
0,
\end{eqnarray*}
and therefore, 
\begin{eqnarray}\label{product of 2nd}
\langle B^{\textbf{j}}(E_{l},E_{k}),
B^{\textbf{j}}(E_{l},\textbf{n})
\rangle=0.
\end{eqnarray}
Thus $\langle A^{\phi}_{\nabla_{E_{l}}^{\perp_{\phi}}H^{\phi}}E_{l},E_{k}\rangle=0$, i.e. 
\begin{eqnarray} \label{The shape Operator}
 A^{\phi}_{\nabla_{E_{l}}^{\perp_{\phi}}H^{\phi}}E_{l} =0.
\end{eqnarray}
Replacing \eqref{Grad H}, \eqref{A the shape operator of M} and \eqref{The shape Operator} in \eqref{Biconservative-CPn} we obtain
\begin{eqnarray*}
4\trace A^{\phi}_{\nabla_{(\cdot)}^{\perp_{\phi}}H^{\phi}}(\cdot)+(1+2q)\grad (\vert H^{\phi}\vert^{2})&=& \frac{4\kappa \kappa ^{'}}{1+2q}E_{1}+\frac{2\kappa \kappa^{'}}{1+2q}E_{1}
\\
&=&
\frac{6 \kappa \kappa ^{'}}{1+2q}E_{1}.
\end{eqnarray*}
In conclusion, 
$\phi$ is biconservative if and only if $\kappa ^{'}=0$, thus $\kappa=constant$.

Now, we are going to prove that $M^{1+2q}$ is a  proper-biharmonic submanifold of 
$\mathbb{C}P^{1+2q}(4)$ if and only if $\kappa^{2}=4$. 

Since $\kappa$ is constant, $M^{1+2q}$ is biconservative and thus the tangential part of the biharmonic equation \eqref{tau-2} holds. Therefore, we need to solve the normal part of the biharmonic equation. As $JH^{\phi}$ is tangent to $M^{1+2q}$, the normal part of the biharmonic equation is \begin{eqnarray}\label{Biharmonic-normal}
-\Delta ^{\perp_{\phi}}H^{\phi}-
\trace B^{\phi}(\cdot, A^{\phi}_{H^{\phi}}(\cdot))+(m+3)H^{\phi}=0
\end{eqnarray} 
(see also \cite{F-L-M-O}).

For the first term of Equation \eqref{Biharmonic-normal}, we have
\begin{eqnarray}\label{normal-Laplacian}
-\Delta ^{\perp_{\phi}}H^{\phi}&=& \nabla_{E_{1}}^{\perp_{\phi}}\nabla_{E_{1}}^{\perp_{\phi}}H^{\phi}-\nabla_{\nabla_{E_{1}}^{M}E_{1}}^{\perp_{\phi}}H^{\phi}
+\sum_{l=2}^{1+2q}\{ \nabla_{E_{l}}^{\perp_{\phi}}\nabla_{E_{l}}^{\perp_{\phi}}H^{\phi}-\nabla_{\nabla_{E_{l}}^{M}E_{l}}^{\perp_{\phi}}H^{\phi} \}\nonumber
\\
&=&
\sum_{l=2}^{1+2q}\{ \nabla_{E_{l}}^{\perp_{\phi}}\nabla_{E_{l}}^{\perp_{\phi}}H^{\phi}-\nabla_{\nabla_{E_{l}}^{\mathbb{C}P^{q}}E_{l}}^{\perp_{\phi}}H^{\phi} \}\nonumber
\\
&=&
\frac{\kappa}{1+2q}\sum_{l=2}^{1+2q}\{ \nabla_{E_{l}}^{\perp_{\phi}}\nabla_{E_{l}}^{\perp_{\phi}}\textbf{n}-\nabla_{\nabla_{E_{l}}^{\mathbb{C}P^{q}}E_{l}}^{\perp_{\phi}}\textbf{n} \}.
\end{eqnarray}
Using (\ref{nablaXl}), we obtain
$$
\nabla_{E_{l}}^{\perp_{\phi}}\nabla_{E_{l}}^{\perp_{\phi}}\textbf{n}=\nabla_{E_{l}}^{\perp_{\phi}}B^{\textbf{j}}(E_{l},\textbf{n}).
$$
Next,
\begin{eqnarray}\label{BjnablaBj}
\nabla_{E_{l}}^{\mathbb{C}P^{1+2q}}B^{\textbf{j}}(E_{l},\textbf{n})=\nabla_{E_{l}}^{\perp_{\phi}}B^{\textbf{j}}(E_{l},\textbf{n})-A^{\phi}_{B^{\textbf{j}}(E_{l},\textbf{n})}E_{l}.
\end{eqnarray}
On the other hand,
\begin{eqnarray}\label{nabla Bj}
\nabla_{E_{l}}^{\mathbb{C}P^{1+2q}}B^{\textbf{j}}(E_{l},\textbf{n})=\nabla_{E_{l}}^{\perp_{\textbf{j}}}B^{\textbf{j}}(E_{l},\textbf{n})-A^{\textbf{j}}_{B^{\textbf{j}}(E_{l},\textbf{n})}E_{l}.
\end{eqnarray}
As 
$\nabla^{\perp_{\textbf{j}}}B^{\textbf{j}}=0$, Equation (\ref{nabla Bj}) becomes 
\begin{eqnarray}\label{nablaBjj}
\nabla_{E_{l}}^{\mathbb{C}P^{1+2q}}B^{\textbf{j}}(E_{l},\textbf{n})
&=& B^{\textbf{j}}(\nabla^{\mathbb{C}P^{q}}_{E_{l}}E_{l},\textbf{n})+B^{\textbf{j}}(E_{l},\nabla^{\mathbb{C}P^{1}\times \mathbb{C}P^{q}}_{E_{l}}\textbf{n})\nonumber
\\
&\ &
-A^{\textbf{j}}_{B^{\textbf{j}}(E_{l},\textbf{n})}E_{l}\nonumber
\\
&=&
B^{\textbf{j}}(\nabla^{\mathbb{C}P^{q}}_{E_{l}}E_{l},\textbf{n})
-A^{\textbf{j}}_{B^{\textbf{j}}(E_{l},\textbf{n})}E_{l}.
\end{eqnarray}
Thus, from Equations (\ref{BjnablaBj}) and  (\ref{nablaBjj}), and using \eqref{product of 2nd}, we obtain 
\begin{eqnarray*}
\nabla_{E_{l}}^{\perp_{\phi}}B^{\textbf{j}}(E_{l},\textbf{n})=B^{\textbf{j}}(\nabla^{\mathbb{C}P^{q}}_{E_{l}}E_{l},\textbf{n})
-\langle A^{\textbf{j}}_{B^{\textbf{j}}(E_{l},\textbf{n})}E_{l},\textbf{n}\rangle \textbf{n}.
\end{eqnarray*}
Hence,
\begin{eqnarray}\label{lap H}
\nabla_{E_{l}}^{\perp_{\phi}}\nabla_{E_{l}}^{\perp_{\phi}}\textbf{n}&=&\nabla_{E_{l}}^{\perp_{\phi}}B^{\textbf{j}}(E_{l},\textbf{n})\nonumber
\\
&=&
B^{\textbf{j}}(\nabla^{\mathbb{C}P^{q}}_{E_{l}}E_{l},\textbf{n})
-\langle B^{\textbf{j}}(E_{l},\textbf{n}), B^{\textbf{j}}(E_{l},\textbf{n})\rangle \textbf{n}\nonumber
\\
&=&
B^{\textbf{j}}(\nabla^{\mathbb{C}P^{q}}_{E_{l}}E_{l},\textbf{n})
-\textbf{n}.
\end{eqnarray}
Replacing \eqref{lap H} in \eqref{normal-Laplacian}, we get
\begin{eqnarray}\label{normal-Laplacian-Calculation}
-\Delta ^{\perp_{\phi}}H^{\phi}&=&
\frac{\kappa}{1+2q}\sum_{l=2}^{1+2q}\{ \nabla_{E_{l}}^{\perp_{\phi}}\nabla_{E_{l}}^{\perp_{\phi}}\textbf{n}-\nabla_{\nabla_{E_{l}}^{\mathbb{C}P^{q}}E_{l}}^{\perp_{\phi}}\textbf{n} \}\nonumber
\\
&=&
\frac{\kappa}{1+2q}\sum_{l=2}^{1+2q}\{B^{\textbf{j}}(\nabla^{\mathbb{C}P^{q}}_{E_{l}}E_{l},\textbf{n})
-\textbf{n}-B^{\textbf{j}}(\nabla^{\mathbb{C}P^{q}}_{E_{l}}E_{l},\textbf{n})\}\nonumber
\\
&=&
\frac{\kappa}{1+2q}\sum_{l=2}^{1+2q}\{-\textbf{n}\}\nonumber
\\
&=&
\frac{-2q\kappa}{1+2q}\textbf{n}.
\end{eqnarray}

Now, we compute 
$\trace B^{\phi}(\cdot, A^{\phi}_{H^{\phi}}(\cdot))$.
From Equations (\ref{A and nabla H1}) and (\ref{nablaXl}) we recall that 
$$
A^{\phi}_{H^{\phi}}E_{1}=\frac{\kappa^{2}}{1+2q}E_{1} \quad \textnormal{and} \quad A^{\phi}_{H^{\phi}}E_{l}=0.
$$
It follows that
\begin{eqnarray}\label{traceBphi}
\trace B^{\phi}(\cdot, A^{\phi}_{H^{\phi}}(\cdot))
&=&
B^{\phi}(E_{1}, A^{\phi}_{H^{\phi}}E_{1})
=
\frac{\kappa^{2}}{1+2q}B^{\phi}(E_{1},E_{1})\nonumber
\\
&=&
\frac{\kappa^{2}}{1+2q}
\{B^{\textbf{i}}(E_{1},E_{1})+B^{\textbf{j}}(E_{1},E_{1})\}\nonumber
\\
&=&
\frac{\kappa^{3}}{1+2q}\textbf{n},
\end{eqnarray}

From (\ref{Biharmonic-normal}), \eqref{normal-Laplacian-Calculation} and (\ref{traceBphi}) we obtain that $M^{1+2q}$ is biharmonic if and only if 
$$
-\frac{2q\kappa}{1+2q}\textbf{n}-\frac{\kappa^{3}}{1+2q}\textbf{n}+\frac{(m+3)\kappa}{1+2q}\textbf{n}=0.
$$
Thus as $\kappa \neq 0$,  we get $\kappa^{2}=4$.

Using the isometry of 
$\mathbb{C}P^{1}(4)$ with the sphere 
$\mathbb{S}^{2}$ of radius $1/2$ and by a standard argument,  we get that a curve $\gamma$ with constant curvature $\kappa=2$ is  a small circle of radius $(1/2)/\sqrt{2}$ of the above sphere $\mathbb{S}^{2}$. Thus, it is proper-biharmonic in 
$\mathbb{C}P^{1}(4)$ (see \cite{C-M-Oniciuc, Caddeo-Montaldo-Piu}).
\end{proof}

\end{document}